\documentclass[runningheads,orivec]{llncs}

\usepackage[T1]{fontenc}
\usepackage{graphicx}
\usepackage[hybrid,underscores=false,codeSpans=false,tightLists = false]{markdown}
\markdownSetup{renderers={headingOne={\section*{#1}},headingTwo={\subsection*{#1}}}}
\usepackage{lmodern}
\usepackage{amsmath}
\usepackage{amsfonts}
\usepackage{mathtools}
\usepackage{enumitem}
\usepackage{booktabs}
\usepackage{makebox}
\usepackage{multirow}
\usepackage{tikz}
\usepackage{pgffor}
\usetikzlibrary{calc,patterns}

\setlist[enumerate,1]{label=(\roman*)}
\setlist{nolistsep}

\usepackage{cite}            
\usepackage[colorlinks,
            citecolor=blue,
            urlcolor=blue,
            linkcolor=blue]{hyperref} 
\usepackage{orcidlink}

\newif\ifBriefFlag
\BriefFlagtrue
\BriefFlagfalse

\newcommand{\EtAl}{\textit{et al.}}

\newcommand{\MacroColor}[1]{#1}

\newcommand{\MyQED}{\mbox{}\hfill$\square$}

\newcounter{MyProblemCounter}

\newcommand{\MyLabel}[1]{%
  \expandafter\xdef\csname MyLabel#1\endcsname{%
    \theMyProblemCounter
  }%
}

\newcommand{\MyProblemTag}[2]{%
  \refstepcounter{MyProblemCounter}%
  \MyProblemTagAgain{#1}{#2}%
}

\newcommand{\MyProblemTagAgain}[2]{%
  \MyLabel{#1}%
  \hypertarget{MyLabel:#1}{}%
  \tag{P$_{\csname MyLabel#1\endcsname}#2$}%
  \label{#1}%
}

\newcommand{\MyEqref}[2]{%
\hyperlink{MyLabel:#1}{%
\text{\expandafter\MyEqrefAux\csname MyLabel#1\endcsname{#2}}%
}%
}

\newcommand{\MyEqrefAux}[2]{%
  (P$_{#1}#2$)%
}

\DeclareMathOperator*{\lexmin}{\MacroColor{lex\,min}}

\DeclareMathOperator*{\Argmin}{\MacroColor{arg\,min}}

\DeclareMathOperator*{\Arglexmin}{\MacroColor{arg\,lex\,min}}

\newcommand{\Size}[1]{\langle#1\rangle}
\newcommand{\SizeTwo}[2]{\langle#1,#2\rangle}
\newcommand{\SizeThree}[3]{\langle#1,#2,#3\rangle}

\newcommand{\Integers}{\MacroColor{\mathbb{Z}}}
\newcommand{\NonNegativeIntegers}{\MacroColor{\mathbb{N}}}
\newcommand{\PositiveIntegers}{\MacroColor{\mathbb{N}_{> 0}}}
\newcommand{\Reals}{\MacroColor{\mathbb{R}}}
\newcommand{\PositiveReals}{\MacroColor{\mathbb{R}_{> 0}}}
\newcommand{\NonNegativeReals}{\MacroColor{\mathbb{R}_{\ge 0}}}

\newcommand{\SecLexFeasibleSet}{\MacroColor{\mathcal{X}}}
\newcommand{\SecLexFeasibleDirections}{\MacroColor{\mathcal{D}}}
\newcommand{\SecLexUniformTestSetNObjs}{\MacroColor{\bar{k}}}
\newcommand{\SecLexUniformTestSetObjCoef}{\MacroColor{\bar{C}}}
\newcommand{\SecLexMaxSubdeterminant}{\MacroColor{\Delta}}

\newcommand{\ZeroVector}{\mathbf{0}}
\newcommand{\OneVector}{\mathbf{1}}

\newcommand{\ForAllInRange}[2]{#1 \in [#2]}

\newcommand{\ComplexityClassStyle}[1]{#1}

\newcommand{\ClassNP}{\MacroColor{\text{\ComplexityClassStyle{NP}}}}
\newcommand{\ClassCoNP}{\MacroColor{\text{\ComplexityClassStyle{coNP}}}}

\newcommand{\DecisionProblemSAT}{\MacroColor{\textsc{SAT}}}
\newcommand{\ValueOfBLPWithSingleUpperLevelVariableName}{\MacroColor{$\textsc{Val}_{\textsc{single}}$}}

\newcommand{\GlobalOptimalProblemName}{\MacroColor{$\textsc{Recog}_{\textsc{single}}^{\textsc{global}}$}}
\newcommand{\CompGlobalOptimalProblemName}{\MacroColor{$\overline{\textsc{Recog}}_{\textsc{single}}^{\textsc{global}}$}}

\newcommand{\GlobalOptimalProblem}{\hyperref[GlobalOptimalProblemAnchor]{\GlobalOptimalProblemName{}}}

\newcommand{\CompGlobalOptimalProblem}{\hyperref[GlobalOptimalProblemAnchor]{\CompGlobalOptimalProblemName{}}}

\newcommand{\ConstraintSense}{\MacroColor{\ge}}
\newcommand{\StrictConstraintSense}{\MacroColor{>}}

\newcommand{\NVariables}{\MacroColor{n}}
\newcommand{\NConstraints}{\MacroColor{m}}

\newcommand{\OptValFunc}{\MacroColor{V}}
\newcommand{\OptSolSet}{\MacroColor{\mathcal{S}}}
\newcommand{\FeasSolSet}{\MacroColor{\mathcal{F}}}
\newcommand{\PartialOptValFunc}{\MacroColor{\psi}}

\newcommand{\SATBooleanFormula}{\MacroColor{F}}

\newcommand{\SATReformulationBinaryVariable}{\MacroColor{x}}
\newcommand{\SATReformulationAuxVariable}{\MacroColor{y}}
\newcommand{\SATReformulationOutputVariable}{\MacroColor{z}}

\newcommand{\SolutionSizeEstimate}{\MacroColor{f_{\mathrm{sol}}}}

\newcommand{\LexLemmaWeight}{\MacroColor{N}}

\newcommand{\LexLemmaWeightPolynomial}{\MacroColor{\psi}}
\newcommand{\LexLemmaWeightIndexFunction}{\MacroColor{\phi}}

\newcommand{\ReformPen}{\MacroColor{e}}
\newcommand{\LexUVar}{\MacroColor{\theta}}

\newcommand{\LexLVarS}{\MacroColor{s}}
\newcommand{\LexLVarT}{\MacroColor{t}}
\newcommand{\LexLVarU}{\MacroColor{u}}
\newcommand{\LexLVarV}{\MacroColor{v}}
\newcommand{\LexLVarW}{\MacroColor{w}}
\newcommand{\LexLVarP}{\MacroColor{p}}
\newcommand{\LexLVarQ}{\MacroColor{q}}

\newcommand{\LexLVarFuncU}{\MacroColor{f_u}}

\newcommand{\LexLVarFuncW}{\MacroColor{f_w}}
\newcommand{\LexLVarFuncInput}{\MacroColor{\theta}}
\newcommand{\ReductionPenWeight}{\MacroColor{M}}
\newcommand{\ReductionLexWeight}{\MacroColor{N}}

\newcommand{\BinaryVector}{\MacroColor{\kappa}}

\newcommand{\SecPricingG}[1]{G_{#1}}

\newcommand{\SecPricingLexUVar}{\MacroColor{\theta}}
\newcommand{\SecPricingLexLVarR}[2]{\MacroColor{r_{#1 #2}}}
\newcommand{\SecPricingLexLVarS}{\MacroColor{s}}
\newcommand{\SecPricingLexLVarT}{\MacroColor{t}}
\newcommand{\SecPricingLexLVarU}{\MacroColor{u}}
\newcommand{\SecPricingUMap}{\MacroColor{\hat{u}}}
\newcommand{\SecPricingLexWeight}{\MacroColor{N}}
\newcommand{\SecPricingSetOfDeltas}{\MacroColor{\bar{\Delta}}}

\newcommand{\SecPricingDualVarOne}{\MacroColor{\alpha}}
\newcommand{\SecPricingDualVarTwo}{\MacroColor{\beta}}
\newcommand{\SecPricingDualVarThree}{\MacroColor{\gamma}}
\newcommand{\SecPricingDualVarFour}{\MacroColor{\delta}}
\newcommand{\SecPricingDualVarFive}{\MacroColor{\eta}}

\newcommand{\SecPricingFrac}{\MacroColor{\bar{\phi}}}

\newcommand{\SecPricingPartialOptValFunc}{\MacroColor{\bar{\psi}}}

\newcommand{\EncodingSize}{size}

\begin{document}

\title{Complexity of Bilevel Linear and Bilevel Pricing Problems with a Single Upper-Level Variable}
%

\author{Nagisa Sugishita\inst{1}\orcidlink{0000-0001-9723-2841} \and
Margarida Carvalho\inst{2}\orcidlink{0000-0002-2344-0960}}
\institute{Département de sciences de la décision, HEC Montréal\\
\email{nagisa.sugishita@hec.ca}
\and
Département d'informatique et de recherche opérationnelle, Université de Montréal\\
\email{carvalho@iro.umontreal.ca}}

\maketitle              
\begin{abstract}
Bilevel linear programming (LP) is one of the simplest classes of bilevel optimization problems, yet it is known to be \ClassNP{}-hard in general. Specifically, determining whether the optimal objective value of a bilevel LP is at least as good as a given threshold, a standard decision version of the problem, is \ClassNP{}-complete. However, this decision problem becomes tractable when either the number of lower-level variables or the number of lower-level constraints is fixed, which prompts the question: What if restrictions are placed on the upper-level problem? In this paper, we address this gap by showing that the decision version of bilevel LP remains \ClassNP{}-complete even when there is only a single upper-level variable, no upper-level constraints (apart from the constraint enforcing optimality of the lower-level decision) and all variables are bounded between 0 and 1. This result implies that fixing the number of variables or constraints in the upper-level problem alone does not lead to tractability in general. On the positive side, we show that there is a polynomial-time algorithm that finds a locally optimal solution of such a rational bilevel LP instance. We also discuss an extension of our result to a bilevel pricing problem. These hardness results rely on auxiliary results of broad interest. Notably, we show that a lexicographic LP can be transformed to a single-objective LP by aggregating objectives with weights of encoding size polynomial in the sizes of cost coefficients and constraint matrix. In particular, these weights are valid uniformly for all right-hand sides.
\end{abstract}

\section{Introduction}

Bilevel programming is a class of hierarchical optimization problems in which a \emph{leader} makes a decision first, and a \emph{follower} subsequently optimizes their own objective subject to the leader's choice. 
In this paper, we focus on the computational complexity of bilevel linear programming (LP), a setting where both the upper- and lower-level problems are LPs.
When the follower's problem admits multiple optimal solutions, the leader's outcome depends on which one is selected, and two standard conventions resolve this ambiguity: in the \emph{optimistic} variant, the follower selects the optimal solution most favorable to the leader, whereas in the \emph{pessimistic} variant, the follower selects the one least favorable to the leader.
We focus primarily on the optimistic variant, but we point out wherever our results carry over to the pessimistic variant.

Even this simplest form of bilevel programming poses serious challenges in theory and practice.
A standard way to analyze the complexity of optimization problems is to consider their decision versions, specifically, determining whether the optimal objective value is equal to or less than a given threshold in a minimization problem.
Jeroslow~\cite{Jeroslow1985} was the first to show that the decision version of bilevel LP is \ClassNP{}-hard.
An alternative proof was given by Ben-Ayed and Blair~\cite{ben1990computational}.
Later, Hansen \EtAl{}~\cite{hansen1992new} showed that the problem remains \ClassNP{}-hard even when restricted to the min-max case.
More recently, Buchheim~\cite{Buchheim2023} proved its membership in \ClassNP{}, thereby establishing its \ClassNP{}-completeness.
Other decision problems have also been studied in the context of bilevel LP: Rodrigues \EtAl{}~\cite{RodriguesEtAl2024} showed that checking whether the problem is unbounded is \ClassNP{}-complete, and Vicente \EtAl{}~\cite{vicente1994descent} showed that recognizing local optimality of a given point is \ClassNP{}-hard.
Recently, Prokopyev and Ralphs~\cite{ProkopyevRalphs2024} studied the complexity of the search of a locally optimal solution.

Given these negative complexity results, a natural line of research is to identify conditions under which bilevel LP becomes tractable. 
Such conditions can guide the design of efficient algorithms for practical instances.
For example, in the context of mixed-integer LP, a classical result by Lenstra Jr ~\cite{lenstra1983integer} establishes that problems with a fixed number of integer variables are solvable in polynomial time.
Similar results have been derived for bilevel LP.
Deng~\cite{deng1998complexity} demonstrated that bilevel LP is solvable in polynomial time when the number of lower-level variables is fixed.
Likewise, the arguments in~\cite{BasuEtAl2021} and~\cite{Buchheim2023} imply polynomial-time solvability when the number of lower-level constraints is fixed.
It was also shown that the network pricing problem can be solved in polynomial time when the number of upper-level variables is fixed, enabling efficient algorithmic strategies~\cite{bui2024asymmetry}.
Related arguments and connections to a pessimistic variant of bilevel LP were very recently discussed in~\cite{ketkov2025note}.

In this paper, we examine the decision version of bilevel LP under a particularly restricted setting: where there is only a single upper-level variable.
We show that even under this restricted structure, the decision problem remains \ClassNP{}-complete.
In particular, we prove a stronger result: the decision problem is \ClassNP{}-complete even when there is only one upper-level variable and no upper-level constraints (other than the constraint enforcing the optimality of the lower-level variable with respect to the lower-level problem), and all variables are constrained to lie in the unit cube (i.e., when the lower-level constraints explicitly include bounds of 0 and 1 on each variable).
Our proof also establishes the \ClassNP{}-hardness of the decision version of the pessimistic formulation of bilevel LP. See the discussion following the proof of Theorem~\ref{theorem:ValueOfBLPWithSingleUpperLevelVariableIsNPHard} (page~\pageref{Rev:Pessimistic}) for details, including the definition of the pessimistic formulation.
We note that in our earlier work~\cite{sugishita2025complexitybilevellinearprogramming}, our discussion on the \ClassNP{}-hardness of the pessimistic variant was incomplete, and this paper fixes the issue.
We also prove that given a candidate solution, the problem of deciding whether it is optimal is also \ClassCoNP{}-complete.
On the positive side, we show that a locally optimal solution can be computed in polynomial time.
We then extend our result to show the \ClassNP{}-completeness of the decision version of a bilevel pricing problem, where the lower-level is an LP.
This is in sharp contrast to the aforementioned result on the network pricing problem~\cite{bui2024asymmetry}, where fixing the number of upper-level variables makes the problem tractable. The computational complexity results presented in this work leverage on auxiliary results of independent interest. Namely, \emph{(i)} we provide a bound on the encoding size of an optimal solution for general bilevel LPs, \emph{(ii)} we describe a polynomial-time transformation of a lexicographic LP into a single-objective LP with aggregation weights independent of the constraints' right-hand-sides, and \emph{(iii)} we prove the validity of the penalty method for removing upper-level constraints, under an assumption absent from existing statements of the
method~\cite{henke2025coupling}, which we show cannot be dropped.

\textbf{Paper Structure.}
This paper is organized as follows.
Section~\ref{sec:MainResult} formally introduces the class of bilevel LP instances, its decision version and its hardness, the main result of this paper.
Section~\ref{sec:AuxiliaryResults} presents auxiliary results that will be used to prove the main result in Section~\ref{sec:ProofOfTheorem1}.
Section~\ref{sec:ComplexityOfVerifyingASolution} is concerned with the complexity of searching for a globally and locally optimal solution.
Section~\ref{sec:Pricing} extends the argument to a bilevel pricing problem.
Section~\ref{SecConclusions} concludes the paper.

\textbf{Conventions.}
We denote the binary encoding size of a rational number $p$ and a rational vector $v$ by $\Size{p}$ and $\Size{v}$, respectively.
Throughout the paper, we refer to the binary encoding size simply as the size.
For any sequence $a_1, \ldots, a_n$, we write $\Size{a_1, \ldots, a_n}$ to denote the sum $\Size{a_1} + \cdots + \Size{a_n}$.
We follow the encoding conventions described in~\cite{Schrijver1998}.
In particular, the size of a rational number $x = p/q$, where $p$ and $q$ are relatively prime integers, is given by $1 + \lceil \log_2(|p| + 1) \rceil + \lceil \log_2(|q| + 1) \rceil$.
We say that a vector, matrix, or optimization instance is \emph{rational} if and only if all its data are rational numbers.
Given an optimization instance $P$, we use $\OptValFunc(P)$, $\FeasSolSet(P)$, and $\OptSolSet(P)$ to denote the optimal objective value, the set of feasible solutions, and the set of optimal solutions, respectively.
We use $\ZeroVector$ (respectively, $\OneVector$) to denote the all-zeros (respectively, all-ones) vector, with the dimension clear from the context.
We write $[n] := \{1, 2, \ldots, n\}$, $\Integers := \{0, \pm1, \pm2, \ldots \}$, $\NonNegativeIntegers := \{0, 1, \ldots \}$.
Given a set $P$, we write $P_{\ge 0}$ ($P_{> 0}$, respectively) to denote the subset of $P$ containing nonnegative (positive, respectively) entries.

\section{Main Result}
\label{sec:MainResult}

In this paper, we study the hardness of rational bilevel LPs with a single upper-level variable.
In fact, we consider a more restrictive class where the upper-level problem has no constraints (other than the variable bound and the optimality with respect to the lower-level problem) and all the variables are bounded to be within the unit cube:
\begin{align}
\MyProblemTag{EqNotationBLP}{}
\min_{x_1, x_2} 
\left\{
c_{1 1} x_1 + c_{1 2}^{\top} x_2
:
\begin{array}{l}
0 \le x_1 \le 1, \\
\displaystyle
x_2 \in \Argmin_{x_2'} 
\left\{ c_{2 2}^{\top} x_2' 
: 
\begin{array}{l}
A_{2 1} x_1 + A_{2 2} x_2' \ge b_2,
\\
0 \le x_2' \le \OneVector
\end{array}
\right\} \notag
\end{array}
\right\},
\end{align}
where $c_{1 1} \in \mathbb{Q}$, $A_{2 1}, b_2 \in \mathbb{Q}^{\NConstraints}$, $c_{1 2}, c_{2 2} \in \mathbb{Q}^{\NVariables}$, and $A_{2 2} \in \mathbb{Q}^{\NConstraints \times \NVariables}$.
The decision version of~\eqref{EqNotationBLP} is to decide if $\OptValFunc\eqref{EqNotationBLP} \le t$, where $t$ is a given rational number.
We will prove the following result.
\begin{theorem}
\label{theorem:ValueOfBLPWithSingleUpperLevelVariableIsNPHard}
The decision version of~\eqref{EqNotationBLP} is NP-complete.
Moreover, this remains true even when $A_{2 1} = -e_1$, where $e_1$ denotes the first standard unit vector, that is, even when the upper-level variable appears in only one lower-level constraint.
\end{theorem}

Independently, \cite{ketkov2025note} proved the first part of Theorem~\ref{theorem:ValueOfBLPWithSingleUpperLevelVariableIsNPHard} in a preprint that appeared after the conference version of this paper \cite{sugishita2025complexitybilevellinearprogramming} was submitted.
Section~\ref{sec:ProofOfTheorem1} provides the proof of Theorem~\ref{theorem:ValueOfBLPWithSingleUpperLevelVariableIsNPHard}, which is based on a reduction from \DecisionProblemSAT{}.
That is, given a Boolean formula $\SATBooleanFormula$, we construct an instance of form~\eqref{EqNotationBLP} whose optimal objective value is $-1$ if $\SATBooleanFormula$ is satisfiable and $0$ otherwise.
One challenge in the reduction presented in Section~\ref{sec:ProofOfTheorem1} is that we need a bilevel LP instance with a single upper-level variable, denoted by $\LexUVar \in [0, 1]$, that encodes $\NVariables$ Boolean values.
We achieve this by having the follower mapping $\LexUVar \in [0, 1]$ to lower-level variables $\LexLVarW \in [0, 1]^{\NVariables}$, so that as $\LexUVar$ varies from $0$ to $1$, the values of $\LexLVarW$ cover all values in $\{0, 1\}^{\NVariables}$.
In particular, we set $\LexLVarW_{i}$ to be the $i$-th digit of the 3-base expansion of $\LexUVar$.
Such a follower is constructed as a lexicographic LP instance, and transformed to a single-objective LP (Section \ref{sec:lexicographic}).
This produces a bilevel LP with upper-level constraints, which we move to the lower-level problem by a penalty term (Section~\ref{subsec:ReformulationOfCouplingConstraints}).
The penalty term must exceed a threshold determined by a bound on the encoding size of an optimal solution to a bilevel LP (Section~\ref{subsec:EstimateOfOptimalSolutionSize}).

\section{Auxiliary Results}
\label{sec:AuxiliaryResults}

This section presents auxiliary results that will be used in Section~\ref{sec:ProofOfTheorem1} to prove Theorem~\ref{theorem:ValueOfBLPWithSingleUpperLevelVariableIsNPHard}.
Section~\ref{subsec:EstimateOfOptimalSolutionSize} provides an estimate on the size of an optimal solution to a bilevel LP.
Section~\ref{subsec:ReformulationOfCouplingConstraints} presents a method for reformulating upper-level constraints as lower-level constraints.
Section~\ref{sec:lexicographic} proves that a lexicographic LP can be transformed to a single-objective LP with aggregated objectives.

\subsection{Estimate of Optimal Solution Size}
\label{subsec:EstimateOfOptimalSolutionSize}

We begin by presenting a bound on the size of an optimal solution to an LP, which will be used to estimate the size of solutions for bilevel LPs.
\begin{lemma}
\label{lemma:estimate-of-primal-linear-inequality}
Let $A, B \in \mathbb{Q}^{\NConstraints \times \NVariables}$, $a, b \in \mathbb{Q}^{\NConstraints}$, and $c \in \mathbb{Q}^{\NVariables}$.  
Let $\sigma$ denote the maximum \EncodingSize{} of the entries in $A$, $B$, $a$, and $b$.  
Consider the LP
\[
\min_{x} \left\{ c^{\top} x : A x \ConstraintSense a, B x \StrictConstraintSense b \right\}.
\]
If an optimal solution exists, then there exists a rational optimal solution $\bar{x}$ of size bounded by a polynomial in $\NVariables$ and $\sigma$.  
That is, there exists a polynomial $\SolutionSizeEstimate$, nondecreasing in each argument, such that $\Size{\bar{x}} \le \SolutionSizeEstimate(\NVariables, \sigma)$.
\end{lemma}
Note that the optimization instance in Lemma~\ref{lemma:estimate-of-primal-linear-inequality} includes strict inequalities.
Nonetheless, we refer to such instances as LPs throughout this paper.
Strict inequalities are used later in the proof of Lemma~\ref{lemma:ZeroViolation}.
\phantomsection\label{Rev:MorePreciseReferenceToSchrijver}
For the proof of Lemma~\ref{lemma:estimate-of-primal-linear-inequality}, see Corollary~10.2 in~\cite{Schrijver1998} (the corollary only discusses weak inequalities, but the extension to strict inequalities is straightforward).
We are now ready to state an analogous result for bilevel LPs.
\begin{lemma}
\label{lemma:SolutionSizeEstimate}
Let
$A_{1 1} \in \mathbb{Q}^{\NConstraints_1 \times \NVariables_1}, A_{1 2} \in \mathbb{Q}^{\NConstraints_1 \times \NVariables_2}, A_{2 1} \in \mathbb{Q}^{\NConstraints_2 \times \NVariables_1}, A_{2 2} \in \mathbb{Q}^{\NConstraints_2 \times \NVariables_2}$, $b_1 \in \mathbb{Q}^{\NConstraints_1}, b_2 \in \mathbb{Q}^{\NConstraints_2}$, $c_{1 1} \in \mathbb{Q}^{\NVariables_1}$ and $c_{1 2}, c_{2 2} \in \mathbb{Q}^{\NVariables_2}$.
Let $\NVariables = \NVariables_1 + \NVariables_2$ and $\sigma$ denote the maximum \EncodingSize{} of the entries in $A_{1 1}, A_{1 2}, A_{2 1}, A_{2 2}$, $b_1$, and $b_2$.
Consider the bilevel LP
\begin{align}
\MyProblemTag{SolutionSizeEstimateBLP}{}
\min_{x_1, x_2} 
\left\{
c_{1 1}^{\top} x_1 + c_{1 2}^{\top} x_2
:
\begin{array}{l}
A_{1 1} x_1 + A_{1 2} x_2 \ge b_1, \notag \\
\displaystyle
x_2 \in \Argmin_{x_2'} \{ c_{2 2}^{\top} x_2' : A_{2 1} x_1 + A_{2 2} x_2' \ge b_2 \} \notag
\end{array}
\right\}.
\end{align}
If an optimal solution exists, then there exists a rational optimal solution $(\bar{x}_1, \bar{x}_2)$ such that $\SizeTwo{\bar{x}_1}{\bar{x}_2} \le \SolutionSizeEstimate(\NVariables, \sigma)$, \phantomsection\label{Rev:DefOfFSol} where $\SolutionSizeEstimate$ is the polynomial defined in Lemma~\ref{lemma:estimate-of-primal-linear-inequality}.
\end{lemma}

\begin{proof}
The feasible set of $\eqref{SolutionSizeEstimateBLP}$ is the union of polyhedra~\cite{BasuEtAl2021}.
Indeed,
\begin{align}
\notag
\FeasSolSet\eqref{SolutionSizeEstimateBLP}
&=
\left\{
\begin{pmatrix}
x_1 \\ x_2
\end{pmatrix}
:
\exists \lambda \in \NonNegativeReals^{\NConstraints_2}
\text{ s.t. }
\begin{array}{l}
A_{1 1} x_1 + A_{1 2} x_2 \ge b_1, \\
A_{2 1} x_1 + A_{2 2} x_2 \ge b_2, \\
\lambda^{\top} ( A_{2 1} x_1 + A_{2 2} x_2 - b_2) = 0, \\
A_{2 2}^{\top} \lambda = c_{2 2}
\end{array}
\right\}
\\
\label{Eq:BilevelFeasibleSetAsUnionOfPolyhedra}
&=
\bigcup_{\omega \in \{0, 1\}^{\NConstraints_2}}
\left\{
\begin{pmatrix}
x_1 \\ x_2
\end{pmatrix}
:
\exists \lambda \in \NonNegativeReals^{\NConstraints_2}
\text{ s.t. }
\begin{array}{ll}
A_{1 1} x_1 + A_{1 2} x_2 \ge b_1, \\
A_{2 1} x_1 + A_{2 2} x_2 \ge b_2, \\
\omega^{\top} (A_{2 1} x_1 + A_{2 2} x_2 - b_2) = 0, \\
(1 - \omega)^{\top} \lambda = 0, \\
A_{2 2}^{\top} \lambda = c_{2 2}
\end{array}
\right\}.
\end{align}
That is, $\FeasSolSet\eqref{SolutionSizeEstimateBLP} = \cup_{\omega \in \Omega} U_{\omega}$, where
$$
\Omega = \{
\omega \in \{0, 1\}^{\NConstraints_2} : \exists \lambda \in \NonNegativeReals^{\NConstraints_2} \text{ s.t. } 
(1 - \omega)^{\top} \lambda = 0, 
A_{2 2}^{\top} \lambda = c_{2 2}
\}
$$
and for each $\omega \in \Omega$,
\begin{align*}
U_{\omega} =
\left\{
\begin{pmatrix}
x_1 \\ x_2
\end{pmatrix}
:
\begin{array}{ll}
A_{1 1} x_1 + A_{1 2} x_2 \ge b_1, \\
A_{2 1} x_1 + A_{2 2} x_2 \ge b_2, \\
\omega^{\top} (A_{2 1} x_1 + A_{2 2} x_2 - b_2) = 0
\end{array}
\right\}.
\end{align*}
Let 
\begin{align}
\MyProblemTag{DecomposedLP}{(\omega)}
\min_{x_1, x_2} \{c_{1 1}^{\top} x_1 + c_{1 2}^{\top} x_2 : (x_1, x_2) \in U_{\omega} \}
\end{align}
and observe that $\OptValFunc\eqref{SolutionSizeEstimateBLP} = \min_{\omega \in \Omega} \OptValFunc\eqref{DecomposedLP}$ and $\OptSolSet\eqref{SolutionSizeEstimateBLP} = \cup_{\omega \in \hat{\Omega}} \OptSolSet\eqref{DecomposedLP}$, where $\hat{\Omega} = \Argmin_{\omega \in \Omega} \OptValFunc\eqref{DecomposedLP}$.
Note that each $U_\omega$ can be defined by inequalities with coefficients of size at most $\sigma$ by disaggregating the constraint involving $\omega$.
Thus, by Lemma~\ref{lemma:estimate-of-primal-linear-inequality}, for each $\omega \in \Omega$, $\OptSolSet\eqref{DecomposedLP}$ contains a rational point of size at most $\SolutionSizeEstimate(\NVariables, \sigma)$ given it is nonempty.
\MyQED
\end{proof}

\phantomsection\label{Rev:RelationToBuchheim}%
We note that this is closely related to the result of~\cite{Buchheim2023}, where it is shown that the decision version of bilevel LP belongs to \ClassNP{}.
Lemma~\ref{lemma:SolutionSizeEstimate} guarantees the existence of a rational optimal solution of polynomial encoding length, provided that an optimal solution exists.
We note that it also guarantees the existence of a feasible solution whose encoding size is polynomial.
Thus, our result implies the membership of the decision version of bilevel LP:
An \ClassNP{} machine can nondeterministically guess a bilevel feasible point of polynomial size whose objective value is less than or equal to the threshold.
We also note that since the feasible set of a bilevel LP instance is a union of polyhedra, if it is feasible and not unbounded, it has an optimal solution.

\subsection{Reformulation of Upper-Level Constraints}
\label{subsec:ReformulationOfCouplingConstraints}

\phantomsection\label{Rev:ComplicationByCoupling}
If the upper-level problem has a constraint that does not involve any lower-level variables, it can be ``moved'' to the lower-level problem without changing the feasible solution set. 
However, if an upper-level constraint involves lower-level variables, special care must be taken. 
In this section, we present our approach for removing such constraints by introducing a penalty term in the objective.
Let $\NConstraints_1, \NConstraints_2 \in \NonNegativeIntegers$, $\NVariables_1, \NVariables_2, \NConstraints_1' \in \PositiveIntegers$, $A_{1 1} \in \mathbb{Q}^{\NConstraints_1 \times \NVariables_1}, A_{1 2} \in \mathbb{Q}^{\NConstraints_1 \times \NVariables_2}, A_{2 1} \in \mathbb{Q}^{\NConstraints_2 \times \NVariables_1}, A_{2 2} \in \mathbb{Q}^{\NConstraints_2 \times \NVariables_2}, \bar{A}_{1 1} \in \mathbb{Q}^{\NConstraints_1' \times \NVariables_1}, \bar{A}_{1 2} \in \mathbb{Q}^{\NConstraints_1' \times \NVariables_2}$, $b_1 \in \mathbb{Q}^{\NConstraints_1}, b_2 \in \mathbb{Q}^{\NConstraints_2}, \bar{b}_1 \in \mathbb{Q}^{\NConstraints_1'}$, $c_{1 1} \in \mathbb{Q}^{\NVariables_1}$, $c_{1 2}, c_{2 2} \in \mathbb{Q}^{\NVariables_2}$, and
\begin{align}
\MyProblemTag{EqPenaltyBefore}{}
\min_{x_1, x_2} \left\{
c_{1 1}^{\top} x_1 + c_{1 2}^{\top} x_2 
:
\begin{array}{l}
A_{1 1} x_1 + A_{1 2} x_2 \ge b_1, \notag \\
\bar{A}_{1 1} x_1 + \bar{A}_{1 2} x_2 \ge \bar{b}_1, \notag \\
\displaystyle
x_2 \in \OptSolSet\eqref{EqPenaltyBeforeLowerLevel}
\end{array}
\right\},
\end{align}
where
\begin{equation}
\MyProblemTag{EqPenaltyBeforeLowerLevel}{(x_1)}
\min_{x_2} \{ c_{2 2}^{\top} x_2 : A_{2 1} x_1 + A_{2 2} x_2 \ge b_2 \}.
\end{equation}
We are interested in ``replacing'' constraints $\bar{A}_{1 1} x_1 + \bar{A}_{1 2} x_2 \ge \bar{b}_1$ in~\eqref{EqPenaltyBefore} with a penalty term in the objective with a help of a single variable $\ReformPen \in \mathbb{R}$ and a real number $M$:
\begin{align}
\MyProblemTag{EqPenaltyAfter}{(M)}
\min_{x_1, x_2, \ReformPen}
\left\{
c_{1 1}^{\top} x_1 + c_{1 2}^{\top} x_2 + M \ReformPen
:
\begin{array}{l}
A_{1 1} x_1 + A_{1 2} x_2 \ge b_1, \\
(x_2, \ReformPen) \in \OptSolSet\eqref{EqPenaltyAfterLowerLevel}
\end{array}
\right\}, \notag 
\end{align}
where
\begin{align}
\MyProblemTag{EqPenaltyAfterLowerLevel}{(x_1)}
\min_{x_2, \ReformPen} 
\left\{ 
c_{2 2}^{\top} x_2
:
\begin{array}{l}
A_{2 1} x_1 + A_{2 2} x_2 \ge b_2, \\
\bar{A}_{1 1} x_1 + \bar{A}_{1 2} x_2 + \ReformPen \OneVector \ge \bar{b}_1, \\
\ReformPen \in I_{\ReformPen}
\end{array}
\right\}
\end{align}
and $I_{\ReformPen}$ is $[0, 1]$ or $[0, \infty)$.
The following result holds for both choices of $I_{\ReformPen}$.
The former is more convenient for our purposes, as we would like all variables to be constrained within $[0, 1]$, whereas the latter is more natural in applications.
\begin{lemma}
\label{lemma:PenaltyReformulation}
Let $I_{\ReformPen}$ be $[0, 1]$ or $[0, \infty)$, $\NVariables = \NVariables_1 + \NVariables_2$, $\sigma$ denote the maximum \EncodingSize{} of entries in $A_{1 1}$, $A_{1 2}$, $A_{2 1}$, $A_{2 2}$, $\bar{A}_{1 1}$, $\bar{A}_{1 2}$, $b_1$, $b_2$, $\bar{b}_1 $, and $c_{\infty} = \max\{ \| c_{1 1} \|_{\infty}, \| c_{1 2} \|_{\infty}, 1 \}$.
Suppose $M \ge 3 \NVariables c_{\infty} 4^{\SolutionSizeEstimate(\NVariables + 1, \sigma)}$, \eqref{EqPenaltyBefore} is feasible and~\eqref{EqPenaltyAfter} is not unbounded from below.
Then,
\begin{enumerate}
\item $\OptValFunc\eqref{EqPenaltyBefore} = \OptValFunc\eqref{EqPenaltyAfter}$;
\item $\OptSolSet\eqref{EqPenaltyAfter} = \{(x_1, x_2, 0) : (x_1, x_2) \in \OptSolSet\eqref{EqPenaltyBefore} \}$.
\end{enumerate}
\end{lemma}
The following remark shows that the assumptions in Lemma~\ref{lemma:PenaltyReformulation} are necessary.

\begin{remark}
\label{Remark:SharpnessOfPenaltyReformulationAssumptions}
The two assumptions in Lemma~\ref{lemma:PenaltyReformulation}, the feasibility of the original instance and the boundedness of the reformulated instance, are necessary.
First, consider the following instance:
\begin{equation}
\MyProblemTag{PenaltyCounterExampleOneBeforePenalty}{}
\min_{u, v} \{ 0 : v \ge 1, v \in \Argmin_{v'} \{ 0 : u = 0, v' = 0 \} \}.
\end{equation}
Instance~\eqref{PenaltyCounterExampleOneBeforePenalty} is infeasible ($\OptValFunc\eqref{PenaltyCounterExampleOneBeforePenalty} = \infty$), but the penalty reformulation
\begin{equation}
\MyProblemTag{PenaltyCounterExampleOneAfterPenalty}{(M)}
\min_{u, v, e} \{ M e : v \in \Argmin_{v', e'} \{ 0 : u = 0, v' = 0, v' + e' \ge 1 \} \}
\end{equation}
has $\FeasSolSet\eqref{PenaltyCounterExampleOneAfterPenalty} = \{ (u, v, e) : u = 0, v = 0, e \ge 1 \}$, and thus $\OptValFunc\eqref{PenaltyCounterExampleOneAfterPenalty} = M$ and $\OptSolSet\eqref{PenaltyCounterExampleOneAfterPenalty} = \{ (0, 0, 1) \}$ for any $M \in \PositiveReals$.

Second, consider
\begin{equation}
\MyProblemTag{PenaltyCounterExampleTwoBeforePenalty}{}
\min_{u, v} \{ u : v \ge 0, v \in \Argmin_{v'} \{ v' : v' \ge -1, v' \ge u \} \}.
\end{equation}
We have $\FeasSolSet\eqref{PenaltyCounterExampleTwoBeforePenalty} = \{ (u, v) : u = v \ge 0 \}$, implying $\OptValFunc\eqref{PenaltyCounterExampleTwoBeforePenalty} = 0$ and $\OptSolSet\eqref{PenaltyCounterExampleTwoBeforePenalty} = \{ (0, 0) \}$.
Its penalty reformulation is
\begin{equation}
\MyProblemTag{PenaltyCounterExampleTwoAfterPenalty}{(M)}
\min_{u, v, e} \{ u + M e : (v, e) \in \Argmin_{v', e'} \{ v' : v' \ge -1, v' \ge u, v' + e' \ge 0 \} \}.
\end{equation}
It can be shown that for any $M \in \PositiveReals$
$$
\FeasSolSet\eqref{PenaltyCounterExampleTwoAfterPenalty} = \{ (u, v, e) : u \in \Reals, v = \max\{-1, u\}, e \in I_e, e \ge -v \}.
$$
In particular, $(u, v, e) = (t, -1, 1) \in \FeasSolSet\eqref{PenaltyCounterExampleTwoAfterPenalty}$ for any $t \le -1$, implying $\OptValFunc\eqref{PenaltyCounterExampleTwoAfterPenalty} = -\infty$ and $\OptSolSet\eqref{PenaltyCounterExampleTwoAfterPenalty} = \emptyset$ for any $M \in \PositiveReals$.
\end{remark}

This penalty-based reformulation also appears in~\cite{henke2025coupling}, where the standing assumption is the solvability of~\eqref{EqPenaltyBefore}, while no assumption is imposed on~\eqref{EqPenaltyAfter}. 
Instance~\eqref{PenaltyCounterExampleTwoBeforePenalty} satisfies their assumptions, and the instance has a unique optimal solution, but its penalty reformulation~\eqref{PenaltyCounterExampleTwoAfterPenalty} is unbounded for every penalty parameter.
Therefore, Theorem~2.2 and Corollary~2.3 of~\cite{henke2025coupling} do not hold on~\eqref{PenaltyCounterExampleTwoBeforePenalty}.\footnote{The last step of the proof of Theorem~2.2 in~\cite{henke2025coupling} is not justified.
The big-$M$ values used to linearize the complementarity conditions of a KKT reformulation depend on the instance at hand, and the ones invoked there ((9) in~\cite{henke2025coupling}) are valid for the KKT reformulation of the original instance.
Once a penalty reformulation is applied, they need not remain valid, so the resulting mixed-integer linear program ((10) in~\cite{henke2025coupling}) need not be the KKT reformulation of the penalized bilevel problem.}
Lemma~\ref{lemma:PenaltyReformulation} identifies an additional assumption, the boundedness of~\eqref{EqPenaltyAfter}, that is required for the penalty reformulation.

We first record two elementary observations, followed by intermediate results to prove Lemma~\ref{lemma:PenaltyReformulation}.
For every $x_1 \in \Reals^{\NVariables_1}$,
\begin{enumerate}[label=(O\arabic*)]
\item
\label{item:ProjectionOfLowerLevel}
if $(x_2, \ReformPen) \in \FeasSolSet\MyEqref{EqPenaltyAfterLowerLevel}{(x_1)}$, then $x_2 \in \FeasSolSet\MyEqref{EqPenaltyBeforeLowerLevel}{(x_1)}$; in particular, $\OptValFunc\MyEqref{EqPenaltyAfterLowerLevel}{(x_1)} \ge \OptValFunc\MyEqref{EqPenaltyBeforeLowerLevel}{(x_1)}$;
\item
\label{item:LiftingToLowerLevel}
if $x_2 \in \FeasSolSet\MyEqref{EqPenaltyBeforeLowerLevel}{(x_1)}$ and $\bar{A}_{1 1} x_1 + \bar{A}_{1 2} x_2 \ge \bar{b}_1$, then $(x_2, 0) \in \FeasSolSet\MyEqref{EqPenaltyAfterLowerLevel}{(x_1)}$.
\end{enumerate}

\begin{lemma}
\label{lemma:PenaltyStepOne}
For every $M \in \PositiveReals$,
$
\mathcal{F}\eqref{EqPenaltyBefore} = \{(x_1,x_2) : (x_1,x_2,0) \in \mathcal{F}\eqref{EqPenaltyAfter}\}.
$
\end{lemma}

\begin{proof}
($\subseteq$)
Let $(\hat{x}_1, \hat{x}_2) \in \FeasSolSet\eqref{EqPenaltyBefore}$.
Then $A_{1 1} \hat{x}_1 + A_{1 2} \hat{x}_2 \ge b_1$, so it suffices to show $(\hat{x}_2, 0) \in \OptSolSet\MyEqref{EqPenaltyAfterLowerLevel}{(\hat{x}_1)}$.
The feasibility of $(\hat{x}_2, 0)$ follows from~\ref{item:LiftingToLowerLevel}, and its optimality follows from the fact that its objective value is
$
c_{2 2}^{\top} \hat{x}_2 = \OptValFunc\MyEqref{EqPenaltyBeforeLowerLevel}{(\hat{x}_1)} \le \OptValFunc\MyEqref{EqPenaltyAfterLowerLevel}{(\hat{x}_1)}
$
by~\ref{item:ProjectionOfLowerLevel}.

($\supseteq$)
Let $(\hat{x}_1, \hat{x}_2, 0) \in \FeasSolSet\eqref{EqPenaltyAfter}$.
Since $(\hat{x}_2, 0) \in \FeasSolSet\MyEqref{EqPenaltyAfterLowerLevel}{(\hat{x}_1)}$, we get $\bar{A}_{1 1} \hat{x}_1 + \bar{A}_{1 2} \hat{x}_2 \ge \bar{b}_1$ and, by~\ref{item:ProjectionOfLowerLevel}, $\hat{x}_2 \in \FeasSolSet\MyEqref{EqPenaltyBeforeLowerLevel}{(\hat{x}_1)}$.
It remains to show that $\hat{x}_2$ is optimal for \MyEqref{EqPenaltyBeforeLowerLevel}{(\hat{x}_1)}.
Suppose not, and let $\tilde{x}_2 \in \FeasSolSet\MyEqref{EqPenaltyBeforeLowerLevel}{(\hat{x}_1)}$ with $c_{2 2}^{\top} \tilde{x}_2 < c_{2 2}^{\top} \hat{x}_2$.
Let $E \ge 1$ be a real number such that
\begin{equation}\label{HintTwo}
\bar{A}_{1 1} \hat{x}_1 + \bar{A}_{1 2} \tilde{x}_2 + E \OneVector \ge \bar{b}_1
\end{equation}
and
$
\bar{x}_2 = ((E - 1) / E) \hat{x}_2 + (1/E) \tilde{x}_2, \bar{\ReformPen} = 1.
$
Then, $(\bar{x}_2, \bar{\ReformPen}) \in \FeasSolSet\MyEqref{EqPenaltyAfterLowerLevel}{(\hat{x}_1)}$, and its objective value $c_{2 2}^{\top} \bar{x}_2$ is strictly smaller than that of $(\hat{x}_2, 0)$, a contradiction to $(\hat{x}_1, \hat{x}_2, 0) \in \FeasSolSet\eqref{EqPenaltyAfter}$.
\MyQED
\end{proof}

From now on we work under the assumptions of Lemma~\ref{lemma:PenaltyReformulation}.
Since~\eqref{EqPenaltyBefore} is feasible, Lemma~\ref{lemma:PenaltyStepOne} shows that~\eqref{EqPenaltyAfter} is feasible as well, and
\begin{equation}
\label{eq:ProofPenaltyReformulationRelaxation}
-\infty < \OptValFunc\eqref{EqPenaltyAfter} \le \OptValFunc\eqref{EqPenaltyBefore} < \infty.
\end{equation}
In particular, both~\eqref{EqPenaltyBefore} and~\eqref{EqPenaltyAfter} have finite optimal values and hence have optimal solutions (see the discussion after Lemma~\ref{lemma:SolutionSizeEstimate}).

\begin{lemma}
\label{lemma:ZeroViolation}
Under the assumptions of Lemma~\ref{lemma:PenaltyReformulation}, for any $(x_1, x_2, \ReformPen) \in \OptSolSet\eqref{EqPenaltyAfter}$, it holds that $\ReformPen = 0$.
\end{lemma}

\begin{proof}
Write $f := \SolutionSizeEstimate(\NVariables + 1, \sigma)$ and suppose, for the sake of contradiction, that $(\hat{x}_1, \hat{x}_2, \hat{\ReformPen}) \in \OptSolSet\eqref{EqPenaltyAfter}$ with $\hat{\ReformPen} > 0$.
As in the proof of Lemma~\ref{lemma:SolutionSizeEstimate},
$
\FeasSolSet\eqref{EqPenaltyAfter} = \cup_{\omega \in \Omega'} U_{\omega}',
$
where each $U_{\omega}' \subseteq \mathbb{R}^{\NVariables + 1}$ is a rational polyhedron defined by inequalities whose coefficients are of \EncodingSize{} at most $\sigma$.
Pick $\hat{\omega} \in \Omega'$ such that $(\hat{x}_1, \hat{x}_2, \hat{\ReformPen}) \in U_{\hat{\omega}}'$ and consider
\begin{equation}
\MyProblemTag{EqPenaltyAfterStrictlyPositiveLP}{(M, \hat{\omega})}
\min_{x_1, x_2, \ReformPen} \{ c_{1 1}^{\top} x_1 + c_{1 2}^{\top} x_2 + M \ReformPen : (x_1, x_2, \ReformPen) \in U_{\hat{\omega}}', \ReformPen > 0 \}.
\end{equation}
Since $(\hat{x}_1, \hat{x}_2, \hat{\ReformPen}) \in \OptSolSet\eqref{EqPenaltyAfterStrictlyPositiveLP}$, we have $\emptyset \not= \OptSolSet\eqref{EqPenaltyAfterStrictlyPositiveLP} \subseteq \OptSolSet\eqref{EqPenaltyAfter}$.
By Lemma~\ref{lemma:estimate-of-primal-linear-inequality},~\eqref{EqPenaltyAfterStrictlyPositiveLP} admits a rational optimal solution $(\tilde{x}_1, \tilde{x}_2, \tilde{\ReformPen})$ with $\SizeThree{\tilde{x}_1}{\tilde{x}_2}{\tilde{\ReformPen}} \le f$; in particular $\| \tilde{x}_1 \|_{\infty}, \| \tilde{x}_2 \|_{\infty} \le 2^{f}$ and, since $\tilde{\ReformPen}$ is positive and rational of \EncodingSize{} at most $f$, $\tilde{\ReformPen} \ge 2^{-f}$.
Let $(\bar{x}_1, \bar{x}_2) \in \OptSolSet\eqref{EqPenaltyBefore}$ be a rational solution such that $\SizeTwo{\bar{x}_1}{\bar{x}_2} \le f$.
For both $(x_1, x_2) = (\tilde{x}_1, \tilde{x}_2)$ and $(x_1, x_2) = (\bar{x}_1, \bar{x}_2)$ we have
$
| c_{1 1}^{\top} x_1 + c_{1 2}^{\top} x_2 |
\le
c_{\infty} ( \NVariables_1 \| x_1 \|_{\infty} + \NVariables_2 \| x_2 \|_{\infty} )
\le
c_{\infty} \NVariables 2^{f},
$
and therefore
\begin{align*}
\OptValFunc\eqref{EqPenaltyAfter} - \OptValFunc\eqref{EqPenaltyBefore}
&=
( c_{1 1}^{\top} \tilde{x}_1 + c_{1 2}^{\top} \tilde{x}_2 + M \tilde{\ReformPen} )
- ( c_{1 1}^{\top} \bar{x}_1 + c_{1 2}^{\top} \bar{x}_2 )
\\
&\ge
M 2^{-f} - 2 c_{\infty} \NVariables 2^{f}
>
0,
\end{align*}
where the second inequality uses $M \ge 3 \NVariables c_{\infty} 4^{f}$.
This contradicts~\eqref{eq:ProofPenaltyReformulationRelaxation}.
\MyQED
\end{proof}

\begin{proof}[Lemma~\ref{lemma:PenaltyReformulation}]
Recall that $\OptSolSet\eqref{EqPenaltyAfter} \ne \emptyset$ and $\OptSolSet\eqref{EqPenaltyBefore} \ne \emptyset$.
Let $(x_1, x_2, \ReformPen) \in \OptSolSet\eqref{EqPenaltyAfter}$.
By Lemma~\ref{lemma:ZeroViolation}, $\ReformPen = 0$, and thus $(x_1, x_2) \in \FeasSolSet\eqref{EqPenaltyBefore}$ by Lemma~\ref{lemma:PenaltyStepOne}.
Its objective value in~\eqref{EqPenaltyBefore} is
$
c_{1 1}^{\top} x_1 + c_{1 2}^{\top} x_2 = \OptValFunc\eqref{EqPenaltyAfter} \le \OptValFunc\eqref{EqPenaltyBefore}
$
by~\eqref{eq:ProofPenaltyReformulationRelaxation}, so $(x_1, x_2) \in \OptSolSet\eqref{EqPenaltyBefore}$ and $\OptValFunc\eqref{EqPenaltyBefore} = \OptValFunc\eqref{EqPenaltyAfter}$.
This proves relation~(i) and the inclusion ``$\subseteq$'' in relation~(ii).

Conversely, let $(x_1, x_2) \in \OptSolSet\eqref{EqPenaltyBefore}$.
By Lemma~\ref{lemma:PenaltyStepOne}, $(x_1, x_2, 0) \in \FeasSolSet\eqref{EqPenaltyAfter}$, and its objective value in~\eqref{EqPenaltyAfter} is $c_{1 1}^{\top} x_1 + c_{1 2}^{\top} x_2 = \OptValFunc\eqref{EqPenaltyBefore} = \OptValFunc\eqref{EqPenaltyAfter}$.
Hence $(x_1, x_2, 0) \in \OptSolSet\eqref{EqPenaltyAfter}$.
\MyQED
\end{proof}

\subsection{Reformulation of Lexicographic Objective}
\label{sec:lexicographic}

The last auxiliary result concerns the transformation of lexicographic optimization using the weighted sum of the objectives.
We first examine lexicographic LP.
Then, we briefly discuss the extension to the lexicographic integer LP.
Lastly, we demonstrate the use of our result on lexicographic LP on a specific instance, that will be used in the proof of our main result.

\subsubsection{Lexicographic LP}

Let $n, m \in \PositiveIntegers$ and $A \in \Integers^{m \times n}$, $c_1, \ldots, c_k \in \Integers^n$.
Consider an optimization instance parametrized by $b$:
$
\lexmin_{x \in \SecLexFeasibleSet(b)} \{ (c_1^{\top} x, \ldots, c_k^{\top} x) \},
$
where $\SecLexFeasibleSet(b) := \{ x \in \Reals^{\NVariables} : A x \ge b \}$.
Here, $\lexmin$ denotes lexicographic minimization: the first objective is given the highest priority, and later objectives are considered only when earlier ones are equal.
We will write the above lexicographic minimization as $\lexmin_{x \in \SecLexFeasibleSet(b)} \{ C x \}$, where $C \in \Integers^{k \times n}$ is the matrix whose $j$-th row equals $c_j$. Formally, for $v, v' \in \Reals^k$ we write $v \prec v'$ if and only if the first nonzero entry of
$v - v'$ is negative, and $v \preceq v'$ if and only if $v \prec v'$ or $v = v'$; the lexicographic
minimization asks for a feasible $x$ with $C x \preceq C x'$ for all feasible $x'$.

It is well-known that the optimal solution set of the above lexicographic LP instance coincides with that of
$
\min_{x \in \SecLexFeasibleSet(b)} w_{\LexLemmaWeight}^{\top} x,
$
where $w_{\LexLemmaWeight} := \sum_{j = 1}^k \LexLemmaWeight^{k - j} c_j$ for some sufficiently large constant $\LexLemmaWeight$.
This approach is discussed in~\cite{sherali1982equivalent,sherali1983preemptive} and was further improved in~\cite{ignizio1984}.
See~\cite{ferro2026} for the latest references on this topic.
However, the weights appearing in all of these results, which follow the approach of~\cite{sherali1983preemptive}, depend on the value of $b$.
In the following, we provide a bound on $\LexLemmaWeight$ that is uniformly valid for any $b$.

Note that, to simplify the argument, in this section we assume that the data $A$ and $C$ are integral.
When they have rational entries, we can always scale them to obtain integral data.
Given an integral matrix $B$, we denote the maximum absolute value of a subdeterminant of $B$ by $\SecLexMaxSubdeterminant(B)$.

\begin{lemma}
\label{Lemma:LexicographicToWeighted}
For any $\LexLemmaWeight \ge 3 \SecLexMaxSubdeterminant(A) \max_j \| c_j \|_1$ and any $b \in \mathbb{R}^{\NConstraints}$, 
$$
\Arglexmin_{x \in \SecLexFeasibleSet(b)} \{ C x \} = \Argmin_{x \in \SecLexFeasibleSet(b)} w_{\LexLemmaWeight}^{\top} x.
$$
\end{lemma}
In what follows, we provide the lemmata allowing us to prove Lemma~\ref{Lemma:LexicographicToWeighted}.

\begin{definition}
\label{Def:LexicographicToWeighted:FeasibleDirection}
Let $b \in \mathbb{R}^{\NConstraints}$ and $x \in \SecLexFeasibleSet(b)$.
We say a vector $t$ is a feasible direction at $x$ if and only if $x + \lambda t \in \SecLexFeasibleSet(b)$ for some $\lambda > 0$.
We denote the set of feasible directions at $x$ by $\SecLexFeasibleDirections(b, x)$.
\end{definition}

\begin{lemma}
\label{Lemma:DecompositionByFeasibleDirection}
For any $b \in \mathbb{R}^{\NConstraints}$ and $x, y \in \SecLexFeasibleSet(b)$, there exists $r \in \NonNegativeIntegers$, $t_1, \ldots, t_r \in \SecLexFeasibleDirections(b, x) \cap \Integers^{\NVariables}$ and $\alpha_1, \ldots, \alpha_r \in \PositiveReals$ such that
$$
y - x = \sum_{i = 1}^r \alpha_i t_i,
\quad
\| t_i \|_{\infty} \le \SecLexMaxSubdeterminant(A), \ \forall i.
$$
\end{lemma}

\begin{proof}
Let $I := \{ i \in [\NConstraints] : A_i x = b_i \}$ be the index set of active constraints, then $\SecLexFeasibleDirections(b, x) = \{ d \in \Reals^{\NVariables} : A_I d \ge 0 \}$.
By Theorem 10.2 of~\cite{Schrijver1998}, $\SecLexFeasibleDirections(b, x)$ is generated by integral vectors of $l_\infty$-norm at most $\SecLexMaxSubdeterminant(A)$.
Now observe that $y - x \in \SecLexFeasibleDirections(b, x)$.
\MyQED
\end{proof}

The next lemma gives a crucial characterization of optimality.
\begin{lemma}[Universal test set]
\label{lemma:TestSet}
Let $\SecLexUniformTestSetNObjs \in \PositiveIntegers$ and $\SecLexUniformTestSetObjCoef \in \Reals^{\SecLexUniformTestSetNObjs \times \NVariables}$.
If $x' \in \SecLexFeasibleSet(b)$ is suboptimal for $\lexmin_{x \in \SecLexFeasibleSet(b)} \{ \SecLexUniformTestSetObjCoef x \}$, there exists $t \in \SecLexFeasibleDirections(b, x') \cap \Integers^{\NVariables}$ such that $\| t \|_{\infty} \le \SecLexMaxSubdeterminant(A)$ and $\SecLexUniformTestSetObjCoef t \prec \ZeroVector$.
\end{lemma}

\begin{proof}
Suppose $x \in \SecLexFeasibleSet(b)$ is suboptimal.
Then, there exists $y$ such that $\SecLexUniformTestSetObjCoef y \prec \SecLexUniformTestSetObjCoef x$.
By Lemma~\ref{Lemma:DecompositionByFeasibleDirection}, we can write
$
y - x = \sum_{i = 1}^r \alpha_i t_i
$
for some $t_1, \ldots, t_r \in \SecLexFeasibleDirections(b, x) \cap \Integers^{\NVariables}$ and $\alpha_1, \ldots, \alpha_r \in \PositiveReals$.
Since $\SecLexUniformTestSetObjCoef (y - x) \prec \ZeroVector$, there exists $i^*$ such that $\SecLexUniformTestSetObjCoef t_{i^*} \prec \ZeroVector$.
\MyQED
\end{proof}

Lemma~\ref{lemma:TestSet} states that, for any cost coefficients, to verify the optimality of a given feasible point one needs only check finitely many feasible directions of small size, and the size bound is independent of $\SecLexUniformTestSetObjCoef, b$ and $x'$.
The next lemma shows that a feasible direction $t$ is an improving direction for $\lexmin_{x \in \SecLexFeasibleSet(b)} \{ C x \}$ if and only if it is so for  $\min_{x \in \SecLexFeasibleSet(b)} \{ w_{\LexLemmaWeight}^{\top} x \}$ given $\LexLemmaWeight$ sufficiently large compared with $\| t \|_{\infty}$.

\begin{lemma}
\label{Lemma:ConsistencyOfImprovingness}
Let $t \in \Integers^{n}$ and suppose $\LexLemmaWeight \ge 3 \| t \|_\infty \| c_j \|_1$ for all $j$.
Then,
$$
C t \prec 0
\qquad
\Longleftrightarrow
\qquad
w_{\LexLemmaWeight}^{\top} t < 0.
$$
\end{lemma}

\begin{proof}
We only show $\Longrightarrow$, as the converse can be shown similarly.
Suppose $C t \prec 0$ and let $j^*$ be such that
$c_j^{\top} t = 0$ for all $j < j^*$ and $c_{j^*}^{\top} t < 0$.
We have
\begin{align*}
w_{\LexLemmaWeight}^{\top} t / \LexLemmaWeight^{k}
=
c_{j^*}^{\top} t / \LexLemmaWeight^{j^*} + \sum_{i = j^* + 1}^{k} c_i^{\top} t / \LexLemmaWeight^i
\le
-1 / \LexLemmaWeight^{j^*} + \sum_{i = j^* + 1}^{k} (\LexLemmaWeight/3) / \LexLemmaWeight^i
< 0.
\end{align*}
\end{proof}

\begin{proof}[Lemma~\ref{Lemma:LexicographicToWeighted}]
Suppose $\LexLemmaWeight \ge 3 \SecLexMaxSubdeterminant(A) \| c_j \|_1$ for all $j$.
Let $x \in \SecLexFeasibleSet(b)$.
Then, by Lemmata~\ref{lemma:TestSet} and~\ref{Lemma:ConsistencyOfImprovingness},
\begin{align*}
&
\text{$x$ is suboptimal for $\lexmin_{x \in \SecLexFeasibleSet(b)} \{ C x \}$}
\\
&
\
\Longleftrightarrow
\
\text{$\exists$ $t \in \SecLexFeasibleDirections(b, x) \cap \Integers^{\NVariables}$ such that $\| t \|_\infty \le \SecLexMaxSubdeterminant(A)$ and $C t \prec \ZeroVector$}
\\
&
\
\Longleftrightarrow
\
\text{$\exists$ $t \in \SecLexFeasibleDirections(b, x) \cap \Integers^{\NVariables}$ such that $\| t \|_\infty \le \SecLexMaxSubdeterminant(A)$ and $w_{\LexLemmaWeight}^{\top} t < 0$}
\\
&
\
\Longleftrightarrow
\
\text{$x$ is suboptimal for $\min_{x \in \SecLexFeasibleSet(b)} \{ w_{\LexLemmaWeight}^{\top} x \}$}.
\end{align*}
\end{proof}

\subsubsection{Lexicographic Integer LP}

In this paper, we are concerned with continuous problems (e.g., LPs).
However, for completeness, we present a self-contained extension to the integer case below, though it is not utilized in our main results.
\begin{lemma}
\label{Lemma:IntegerLexicographicToWeighted}
For any $\LexLemmaWeight \ge 3 \NVariables \SecLexMaxSubdeterminant(A) \max_j \| c_j \|_1$ and any $b \in \mathbb{R}^{\NConstraints}$, 
$$
\Arglexmin_{x \in \SecLexFeasibleSet'(b)} \{ C x \} = \Argmin_{x \in \SecLexFeasibleSet'(b)} w_{\LexLemmaWeight}^{\top} x,
$$
where $\SecLexFeasibleSet'(b) := \{ x \in \Integers^{\NVariables} : A x \ge b \} = \SecLexFeasibleSet(b) \cap \Integers^{\NVariables}$ for all $b \in \Reals^{\NConstraints}$.
\end{lemma}
The proof is given in Appendix~\ref{Sec:Proof:Lemma:IntegerLexicographicToWeighted}, which closely follows that of Lemma~\ref{Lemma:LexicographicToWeighted}.
In particular, the proof is based on the Graver basis and an integer version of the universal test set.

\subsubsection{Example}

Lastly, we demonstrate the use of Lemma~\ref{Lemma:LexicographicToWeighted} on an example instance.
This result will be used in the proof of the main result in Section~\ref{sec:ProofOfTheorem1}.

Let $\NVariables \in \PositiveIntegers$ and consider a lexicographic LP parametrized by $\LexUVar \in \Reals$:
\begin{alignat}{3}
\MyProblemTag{EqSATBilevelLexProgrammingLowerLevel}{(\NVariables, \LexUVar)}
\lexmin_{\LexLVarS, \LexLVarT, \LexLVarU, \LexLVarV, \LexLVarW} 
\left\{
(
\LexLVarS_{1} + \LexLVarT_{1} + \LexLVarU_{1}, \
\ldots, \
\LexLVarS_{n} + \LexLVarT_{n} + \LexLVarU_{n}, \
\OneVector^{\top} \LexLVarV
)
:
{\begin{array}{llllll}
&
\LexLVarU_{1} \ge \LexUVar,
\\
&
\LexLVarS \ge (3/2) \LexLVarU - (1/2) \cdot \OneVector,
\\
&
\LexLVarT \ge (3/2) \LexLVarU - \OneVector,
\\
&
\LexLVarW = 2 (\LexLVarS - \LexLVarT),
\\
&
-\LexLVarV \le \LexLVarW - (1/2) \cdot \OneVector \le \LexLVarV, 
\\
&
\LexLVarU_{i + 1} \ge 3 \LexLVarU_i - 2 \LexLVarW_i,
\ \
\forall \ForAllInRange{i}{\NVariables - 1},
\\
&
\ZeroVector \le \LexLVarS, \LexLVarT, \LexLVarU, \LexLVarV, \LexLVarW \le \OneVector
\end{array}}
\right\}
\end{alignat}
with $\LexLVarS, \LexLVarT, \LexLVarU, \LexLVarV, \LexLVarW \in \mathbb{R}^{\NVariables}$.
Let $\ReductionLexWeight = 2^{\LexLemmaWeightPolynomial(\NVariables)}$, where $\LexLemmaWeightPolynomial$ is a polynomial, be a sufficiently large constant as stated in Lemma~\ref{Lemma:LexicographicToWeighted} to aggregate the objectives in~\eqref{EqSATBilevelLexProgrammingLowerLevel}.
Then, one can aggregate the objectives as
\begin{alignat}{3}
\MyProblemTag{EqSATBilevelWeightedProblemLowerLevelTwo}{(\NVariables, \LexUVar)}
\min_{\LexLVarS, \LexLVarT, \LexLVarU, \LexLVarV, \LexLVarW} 
\left\{
\sum_{i = 1}^n 
\ReductionLexWeight^{n - i + 1} (\LexLVarS_{i} + \LexLVarT_{i} + \LexLVarU_{i}) 
+
\OneVector^{\top} \LexLVarV
:
(\LexLVarS, \LexLVarT, \LexLVarU, \LexLVarV, \LexLVarW) \in \FeasSolSet\eqref{EqSATBilevelLexProgrammingLowerLevel}
\right\}.
\end{alignat}
By Lemma~\ref{Lemma:LexicographicToWeighted}, for each $\LexUVar \in [0, 1]$, $\OptSolSet\eqref{EqSATBilevelLexProgrammingLowerLevel} = \OptSolSet\eqref{EqSATBilevelWeightedProblemLowerLevelTwo}$ and they are a singleton as given below.
\begin{lemma}
\label{lemma:LexLowerLevelSolution}
\mbox{}
For any $\LexUVar \in [0, 1]$, it holds that
\begin{align}
\notag
&
\OptSolSet\eqref{EqSATBilevelLexProgrammingLowerLevel} =
\\
\label{EqLexLowerLevelSolution}
&
\
\left\{
\begin{pmatrix}
\LexLVarS \\ \LexLVarT \\ \LexLVarU \\ \LexLVarV \\ \LexLVarW
\end{pmatrix}
:
\begin{array}{llllll}
&
\LexLVarU_{i} = \LexLVarFuncU^{i - 1}(\LexUVar),
&&
\LexLVarW_i = \LexLVarFuncW(\LexLVarU_{i}),
&&
\forall \ForAllInRange{i}{\NVariables},
\\
&
\LexLVarS_i = [(3/2) \LexLVarU_i - 1/2]^+, 
\quad
&&
\LexLVarT_i = [(3/2) \LexLVarU_i - 1]^+,
\quad
&&
\forall \ForAllInRange{i}{\NVariables},
\\
&
\LexLVarV_i = |\LexLVarW_i - 1/2|,
&&
&&
\forall \ForAllInRange{i}{\NVariables},
\end{array}
\right\},
\end{align}
where
\begin{align*}
\LexLVarFuncW(\LexLVarFuncInput{}') = 
\begin{cases}
0, & 0 \le \LexLVarFuncInput{}' < 1/3, \\
3 \LexLVarFuncInput{}' - 1, & 1/3 \le \LexLVarFuncInput{}' < 2/3, \\
1, & 2/3 \le \LexLVarFuncInput{}' \le 1,
\end{cases}
\qquad
\LexLVarFuncU(\LexLVarFuncInput{}') = 
\begin{cases}
3 \LexLVarFuncInput{}', & 0 \le \LexLVarFuncInput{}' < 1/3, \\
-3 \LexLVarFuncInput{}' + 2, & 1/3 \le \LexLVarFuncInput{}' < 2/3, \\
3 \LexLVarFuncInput{}' - 2, & 2/3 \le \LexLVarFuncInput{}' \le 1,
\end{cases}
\end{align*}
$\LexLVarFuncU^i$ is the composition of $\LexLVarFuncU$ for $i$ times, and $[ a ]^+ = \max\{a, 0\}$.
\end{lemma}

\begin{proof}
Fix $\LexUVar \in [0, 1]$.
For any $(\LexLVarS, \LexLVarT, \LexLVarU, \LexLVarV, \LexLVarW) \in \FeasSolSet\eqref{EqSATBilevelLexProgrammingLowerLevel}$, we have
$$
\LexLVarS_{1} \ge [(3/2) \LexUVar - 1/2]^+,
\LexLVarT_{1} \ge [(3/2) \LexUVar - 1]^+,
\LexLVarU_{1} \ge \LexUVar.
$$
Thus, the value of the first objective $\LexLVarS_{1} + \LexLVarT_{1} + \LexLVarU_{1}$ is bounded by $[(3/2) \LexUVar - 1/2]^+ + [(3/2) \LexUVar - 1]^+ + \LexUVar$ from below.
This lower bound is attained if and only if there is $(\LexLVarS, \LexLVarT, \LexLVarU, \LexLVarV, \LexLVarW)$ in $\FeasSolSet\eqref{EqSATBilevelLexProgrammingLowerLevel}$ satisfying $\LexLVarU_{1} = \LexUVar$, $\LexLVarS_{1} = [(3/2) \LexUVar - 1/2]^+$ and $\LexLVarT_{1} = [(3/2) \LexUVar - 1]^+$.
The point specified in~\eqref{EqLexLowerLevelSolution} satisfies this condition, and thus this lower bound is attained.
Therefore, the optimality with respect to the first objective $\LexLVarS_{1} + \LexLVarT_{1} + \LexLVarU_{1}$ implies $\LexLVarU_{1} = \LexUVar$, $\LexLVarS_{1} = [(3/2) \LexLVarU_{1} - 1/2]^+$, $\LexLVarT_{1} = [(3/2) \LexLVarU_{1} - 1]^+$, and it follows that $\LexLVarW_{1} = \LexLVarFuncW(\LexLVarU_{1})$ and $\LexLVarU_{2} \ge \LexLVarFuncU(\LexLVarU_{1})$.
Now, consider the problem obtained by fixing $\LexLVarU_{1} = \LexUVar$, $\LexLVarS_{1} = [(3/2) \LexUVar - 1/2]^+$ and $\LexLVarT_{1} = [(3/2) \LexUVar - 1]^+$ in~\eqref{EqSATBilevelLexProgrammingLowerLevel}.
It is equivalent to \MyEqref{EqSATBilevelLexProgrammingLowerLevel}{(\NVariables - 1, \LexLVarFuncU(\LexLVarU_{1}))}, and repeating the above argument, the optimality of the second objective $\LexLVarS_{2} + \LexLVarT_{2} + \LexLVarU_{2}$ implies $\LexLVarU_{2} = \LexLVarFuncU(\LexUVar)$, $\LexLVarS_{2} = [(3/2) \LexLVarU_{2} - 1/2]^+$, $\LexLVarT_{2} = [(3/2) \LexLVarU_{2} - 1]^+$ and $\LexLVarW_{2} = \LexLVarFuncW(\LexLVarU_{2})$.
By repeating the same argument, we can show that the optimality of the first $\NVariables$ objectives implies
$$
\begin{array}{llllll}
&
\LexLVarU_{i} = \LexLVarFuncU^{i - 1}(\LexUVar),
&&
\LexLVarW_i = \LexLVarFuncW(\LexLVarU_{i}),
&&
\forall \ForAllInRange{i}{\NVariables},
\\
&
\LexLVarS_i = [(3/2) \LexLVarU_i - 1/2]^+, 
\quad
&&
\LexLVarT_i = [(3/2) \LexLVarU_i - 1]^+,
\quad
&&
\forall \ForAllInRange{i}{\NVariables}.
\end{array}
$$
Now, consider the problem by fixing variables in~\eqref{EqSATBilevelLexProgrammingLowerLevel} as above.
Then, the resulting problem is
\begin{alignat}{3}
\notag
\min_{\LexLVarS, \LexLVarT, \LexLVarU, \LexLVarV, \LexLVarW} 
\left\{
\OneVector^{\top} \LexLVarV
:
\begin{array}{llllll}
&
\LexLVarU_{i} = \LexLVarFuncU^{i - 1}(\LexUVar),
&&
\LexLVarW_i = \LexLVarFuncW(\LexLVarU_{i}),
&&
\forall \ForAllInRange{i}{\NVariables},
\\
&
\LexLVarS_i = [(3/2) \LexLVarU_i - 1/2]^+, 
\quad
&&
\LexLVarT_i = [(3/2) \LexLVarU_i - 1]^+,
\quad
&&
\forall \ForAllInRange{i}{\NVariables},
\\
&
-\LexLVarV \le \LexLVarW - (1/2) \cdot \OneVector \le \LexLVarV, 
&&
\ZeroVector \le \LexLVarV \le \OneVector
\end{array}
\right\}.
\end{alignat}
Thus, the optimality of the last objective implies $\LexLVarV_i = |\LexLVarW_i - 1/2|$ for any $i$.
\MyQED
\end{proof}

For any binary vector $\BinaryVector \in \{0, 1\}^{\NVariables}$, $\LexUVar = \sum_{i = 1}^{\NVariables} 2 \BinaryVector_i/3^{i}$ and $(\LexLVarS, \LexLVarT, \LexLVarU, \LexLVarV, \LexLVarW) \in \OptSolSet\eqref{EqSATBilevelLexProgrammingLowerLevel}$ implies $\LexLVarW_i = \BinaryVector_i$ for $i \in [\NVariables]$.
Plots illustrating the relationship between $\LexLVarU$, $\LexLVarW$ and $\LexUVar$ are provided in Figure~\ref{fig:plot}.
As $\LexUVar$ varies from $0$ to $1$, each component of $\LexLVarW$ takes values between $0$ and $1$. 
Collectively, the vector $\LexLVarW$ takes all binary values in $\{0, 1\}^{\NVariables}$.
Furthermore, for any $\LexUVar \in [0, 1]$ and $i \in [\NVariables]$, $\LexLVarV_i \ge 1/2$ if and only if $\LexLVarW_i \in \{0, 1\}$.

\begin{figure}[btp]
\centering
\begin{tikzpicture}[
xscale=4,
yscale=4,
point/.style={circle,inner sep=0,minimum size=1mm,fill=black},
nodepadding/.style={circle,inner sep=0,minimum size=6mm},
unselected/.style={draw=black, densely dashed},
selected/.style={draw=black},
first/.style={ultra thick},
second/.style={ultra thick,blue,dashed},
third/.style={ultra thick,red,dotted},
]

\begin{scope}[shift={(0,0)}]
\foreach \i in {1,...,9} {
    \draw [very thin,gray] (\i/9,-0.0) -- (\i/9,1.1);
    \draw [very thin,gray] (-0.0,\i/9) -- (1.1,\i/9);
}
   
\draw[->] (-0.2,0) -- (1.2,0) node[above] {$\LexUVar$};
\draw[->] (0,-0.2) -- (0,1.2) node[above] {$\LexLVarU$};

\begin{scope}
\clip(0,0) rectangle (1,1);
\draw [first] (0,0) -- (1,1);
\draw [second] (0,0) -- (1/3,1) -- (2/3, 0) -- (1,1);
\draw [third] (0,0) -- (1/9,1) -- (2/9,0) -- (3/9,1) -- (4/9,0) -- (5/9,1) -- (6/9,0) -- (7/9,1) -- (8/9, 0) -- (9/9,1);
\end{scope}

\end{scope}

\begin{scope}[shift={(1.5,0)}]
\foreach \i in {1,...,9} {
    \draw [very thin,gray] (\i/9,-0.0) -- (\i/9,1.1);
    \draw [very thin,gray] (-0.0,\i/9) -- (1.1,\i/9);
}
   
\draw[->] (-0.2,0) -- (1.2,0) node[above] {$\LexUVar$};
\draw[->] (0,-0.2) -- (0,1.2) node[above] {$\LexLVarW$};

\begin{scope}
\draw [first] (0,0) -- (1/3,0) -- (2/3, 1) -- (1,1);
\draw [second] (0,0) -- (1/9,0) -- (2/9, 1) -- (3/9,1) -- (4/9,1) -- (5/9,0) -- (6/9, 0) -- (7/9,0) -- (8/9, 1) -- (9/9,1);
\draw [third]
(0/27,0) -- (1/27,0) -- (2/27, 1) -- (3/27,1) -- (4/27,1) -- (5/27,0) -- (6/27, 0) -- (7/27,0) -- (8/27, 1) -- (9/27,1) --
(10/27,1) -- (11/27,0) -- (12/27, 0) -- (13/27,0) -- (14/27,1) -- (15/27,1) -- (16/27, 1) -- (17/27,0) -- (18/27, 0) -- (19/27,0) --
(20/27,1) -- (21/27,1) -- (22/27, 1) -- (23/27,0) -- (24/27,0) -- (25/27,0) -- (26/27, 1) -- (27/27,1)
;
\end{scope}

\end{scope}

\end{tikzpicture}
\caption{Plot of $\LexLVarU$ and $\LexLVarW$ in $\OptSolSet\eqref{EqSATBilevelLexProgrammingLowerLevel}$.
The solid black lines represent $\LexLVarU_1$ and $\LexLVarW_1$, the dashed blue lines represent $\LexLVarU_2$ and $\LexLVarW_2$, and the dotted red lines represent $\LexLVarU_3$ and $\LexLVarW_3$.
}
\label{fig:plot}
\end{figure}

\begin{remark}
\label{Remark:PolyMagnitudeWeights}
Instance~\eqref{EqSATBilevelWeightedProblemLowerLevelTwo} has constants of exponential magnitude.
However, this can be avoided by introducing auxiliary variables.
Recall $\ReductionLexWeight = 2^{\LexLemmaWeightPolynomial(\NVariables)}$ for some polynomial $\LexLemmaWeightPolynomial$.
Observe that by rescaling the objective,~\eqref{EqSATBilevelWeightedProblemLowerLevelTwo} can be rewritten as
\begin{alignat*}{3}
\min_{\LexLVarS, \LexLVarT, \LexLVarU, \LexLVarV, \LexLVarW}
\left\{
\sum_{i = 1}^n 
(\LexLVarS_{i} + \LexLVarT_{i} + \LexLVarU_{i}) / \ReductionLexWeight^{i - 1}
+
\OneVector^{\top} \LexLVarV / \ReductionLexWeight^{n}
:
(\LexLVarS, \LexLVarT, \LexLVarU, \LexLVarV, \LexLVarW) \in \FeasSolSet\eqref{EqSATBilevelLexProgrammingLowerLevel}
\right\}.
\end{alignat*}
Now, introduce auxiliary variables $\LexLVarP \in \Reals^{\LexLemmaWeightIndexFunction(\NVariables) \times n}$ and $\LexLVarQ \in \Reals^{\LexLemmaWeightIndexFunction(\NVariables + 1)}$, where $\LexLemmaWeightIndexFunction(i) := (i - 1) \LexLemmaWeightPolynomial(n) + 1$, and define
\begin{alignat}{3}
\MyProblemTag{EqSATBilevelWeightedProblemLowerLevelTwoPolyWeights}{(\NVariables, \LexUVar)}
\min_{\LexLVarS, \LexLVarT, \LexLVarU, \LexLVarV, \LexLVarW, \LexLVarP, \LexLVarQ} 
\left\{
\sum_{i = 1}^n 
3 \LexLVarP_{\LexLemmaWeightIndexFunction(i) \, i}
+
n \LexLVarQ_{\LexLemmaWeightIndexFunction(n + 1)}
:
\begin{array}{lllllll}
&
\LexLVarP_{1 j} = (\LexLVarS_j + \LexLVarT_j + \LexLVarU_j) / 3,
&&
\forall \ForAllInRange{j}{\NVariables},
\\
&
\LexLVarP_{i + 1 \, j} = \LexLVarP_{i j} / 2,
\qquad\qquad\quad
\quad
&&
\forall \ForAllInRange{i}{\LexLemmaWeightIndexFunction(\NVariables) - 1}, \ForAllInRange{j}{\NVariables},
\\
&
\LexLVarQ_{1} = \OneVector^{\top} \LexLVarV / n,
\\
&
\LexLVarQ_{i + 1} = \LexLVarQ_{i} / 2,
&&
\forall \ForAllInRange{i}{\LexLemmaWeightIndexFunction(\NVariables + 1) - 1},
\\
&
0 \le \LexLVarP_{i j} \le 1,
\ZeroVector \le \LexLVarQ \le \OneVector,
&&
\forall \ForAllInRange{i}{\LexLemmaWeightIndexFunction(\NVariables)}, \ForAllInRange{j}{\NVariables},
\\
&
(\LexLVarS, \LexLVarT, \LexLVarU, \LexLVarV, \LexLVarW) \in \FeasSolSet\eqref{EqSATBilevelLexProgrammingLowerLevel}
\end{array}
\right\}.
\end{alignat}
It can be shown that $(\LexLVarS, \LexLVarT, \LexLVarU, \LexLVarV, \LexLVarW, \LexLVarP, \LexLVarQ) \in \FeasSolSet\eqref{EqSATBilevelWeightedProblemLowerLevelTwoPolyWeights}$ if and only if $(\LexLVarS, \LexLVarT, \LexLVarU, \LexLVarV, \LexLVarW) \in \FeasSolSet\eqref{EqSATBilevelWeightedProblemLowerLevelTwo}$, and $(\LexLVarS, \LexLVarT, \LexLVarU, \LexLVarV, \LexLVarW, \LexLVarP, \LexLVarQ) \in \FeasSolSet\eqref{EqSATBilevelWeightedProblemLowerLevelTwoPolyWeights}$ implies $3 \LexLVarP_{\LexLemmaWeightIndexFunction(i) \, i} = (\LexLVarS_i + \LexLVarT_i + \LexLVarU_i) / 2^{(i - 1) \LexLemmaWeightPolynomial(n)} = (\LexLVarS_i + \LexLVarT_i + \LexLVarU_i) / \ReductionLexWeight^{i - 1}$ for $i = 1, \ldots, n$ and $n \LexLVarQ_{\LexLemmaWeightIndexFunction(n + 1)} = \OneVector^{\top} \LexLVarV / \ReductionLexWeight^{n}$.
Thus, the projection of $\OptSolSet\eqref{EqSATBilevelWeightedProblemLowerLevelTwoPolyWeights}$ onto the space of $(\LexLVarS, \LexLVarT, \LexLVarU, \LexLVarV, \LexLVarW)$ equals $\OptSolSet\eqref{EqSATBilevelWeightedProblemLowerLevelTwo}$.

The elimination of large coefficients with auxiliary variables appears to be a recurring technique in the literature. 
For example, see the discussion following Theorem 5 in \cite{Buchheim2023}, as well as Lemma 8 (Coefficient Reduction) in \cite{knop2020tight}.
\end{remark}

\section{Proof of Theorem~\ref{theorem:ValueOfBLPWithSingleUpperLevelVariableIsNPHard}}
\label{sec:ProofOfTheorem1}

In this section, we prove Theorem~\ref{theorem:ValueOfBLPWithSingleUpperLevelVariableIsNPHard} by transforming a Boolean formula $\SATBooleanFormula$ into a bilevel LP instance (in maximization form) whose optimal objective value is $1$ if $\SATBooleanFormula$ is satisfiable, and $0$ otherwise.
Let \( \SATBooleanFormula \) be a Boolean formula with \( n \) variables.
Following~\cite{Jeroslow1985}, \( F \) can be represented by a system of linear inequalities:
\begin{equation}
\label{Eq:JeroslowInequality}
A_{\SATBooleanFormula} \SATReformulationBinaryVariable{}
+
B_{\SATBooleanFormula} \SATReformulationAuxVariable{}
+
C_{\SATBooleanFormula} \SATReformulationOutputVariable{}
\geq
d_{\SATBooleanFormula}, 
\ 
\SATReformulationBinaryVariable{} \in [0,1]^n,
\SATReformulationAuxVariable{} \in [0,1]^{n'},
\SATReformulationOutputVariable{} \in [0,1],
\end{equation}
where $n'$ is a nonnegative integer, \( A_{\SATBooleanFormula} \), \( B_{\SATBooleanFormula} \), \( C_{\SATBooleanFormula} \) and \( d_{\SATBooleanFormula} \) are matrices and a vector of conforming sizes whose entries are $-1$, $0$ or $1$, such that the following holds.
\begin{itemize}
\item 
For each $\SATReformulationBinaryVariable{} \in [0, 1]^n$, there exists $(\SATReformulationAuxVariable{}, \SATReformulationOutputVariable{})$ satisfying~\eqref{Eq:JeroslowInequality};
\item 
For each binary $\SATReformulationBinaryVariable{}$, $(\SATReformulationAuxVariable{}, \SATReformulationOutputVariable{})$ satisfying~\eqref{Eq:JeroslowInequality} are unique and binary, and $\SATReformulationOutputVariable{} = 1$ if and only if $\SATReformulationBinaryVariable{}$ corresponds to a truth assignment satisfying $\SATBooleanFormula$.
\end{itemize}
Variable $\SATReformulationAuxVariable{}$ corresponds to the truth values of subformulae in $\SATBooleanFormula$, while $\SATReformulationOutputVariable{}$ corresponds to that of $\SATBooleanFormula$.
For example, if $F = (x_1 \vee x_2) \wedge x_3$, we introduce an auxiliary variable $y_1$ that corresponds to the subformula $(x_1 \vee x_2)$ and define
\begin{align*}
& y_1 \ge x_1, y_1 \ge x_2, y_1 \le x_1 + x_2, \\
& z \ge x_3 + y_1 - 1, z \le x_3, z \le y_1, \\
& 0 \le x_1, x_2, x_3, y_1, z \le 1.
\end{align*}
In the following, we say that a binary vector satisfies \( \SATBooleanFormula \) if and only if the corresponding truth assignment satisfies \( \SATBooleanFormula \).
Note that one can formulate the satisfiability of $\SATBooleanFormula$ as an optimization problem:
\begin{equation}
\MyProblemTag{EqJeroslowInequalityOptimization}{(\SATBooleanFormula)}
\max_{\SATReformulationBinaryVariable{}, \SATReformulationAuxVariable{}, \SATReformulationOutputVariable{}} \{
\SATReformulationOutputVariable{} :
A_{\SATBooleanFormula} \SATReformulationBinaryVariable{}
+ B_{\SATBooleanFormula} \SATReformulationAuxVariable{}
+ C_{\SATBooleanFormula} \SATReformulationOutputVariable{}
\ge
d_{\SATBooleanFormula},
x \in \{0, 1\}^{n},
y \in [0, 1]^{n'},
z \in [0, 1]
\}.
\end{equation}
We have $\OptValFunc\eqref{EqJeroslowInequalityOptimization} = 1$ if $\SATBooleanFormula$ is satisfiable, and $0$ otherwise.

Now, consider the following bilevel LP:
\begin{align}
\MyProblemTag{EqBLPWithLinking}{(\SATBooleanFormula)}
\max_{\substack{\LexUVar, \LexLVarS, \LexLVarT, \LexLVarU, \LexLVarV,\\\LexLVarW, \SATReformulationBinaryVariable{}, \SATReformulationAuxVariable{}, \SATReformulationOutputVariable{}}}
\{
z
:
\LexLVarV \ge (1/2) \cdot \OneVector, \
(
\LexLVarS, \LexLVarT, \LexLVarU, \LexLVarV, \LexLVarW, x, y, z
)
\in
\OptSolSet\eqref{EqBLPWithLinkingLowerLevel}
\},
\end{align}
where
\begin{align}
\MyProblemTag{EqBLPWithLinkingLowerLevel}{(\LexUVar, \SATBooleanFormula)}
\min_{\substack{\LexLVarS, \LexLVarT, \LexLVarU, \LexLVarV,\\\LexLVarW, \SATReformulationBinaryVariable{}, \SATReformulationAuxVariable{}, \SATReformulationOutputVariable{}}}
\left\{
\displaystyle
\sum_{i = 1}^n 
\ReductionLexWeight^{n - i + 1} (\LexLVarS_{i} + \LexLVarT_{i} + \LexLVarU_{i})
+
\OneVector^{\top} \LexLVarV
:
\begin{array}{l}
\begin{array}[t]{l}
(
\LexLVarS, \LexLVarT, \LexLVarU, 
\LexLVarV, \LexLVarW
)
\in \FeasSolSet\eqref{EqSATBilevelLexProgrammingLowerLevel},
\\
A_{\SATBooleanFormula} \SATReformulationBinaryVariable{}
+ B_{\SATBooleanFormula} \SATReformulationAuxVariable{}
+ C_{\SATBooleanFormula} \SATReformulationOutputVariable{}
\ge d_{\SATBooleanFormula},
\\
\SATReformulationBinaryVariable{} = \LexLVarW,
\\
\ZeroVector \le \SATReformulationBinaryVariable{}, \SATReformulationAuxVariable{}, \SATReformulationOutputVariable{} \le \OneVector
\end{array}
\end{array}
\right\}
\end{align}
where $\ReductionLexWeight$ is the constant to aggregate the objective in~\eqref{EqSATBilevelWeightedProblemLowerLevelTwo}.

\begin{lemma}
\label{lemma:PenToLex}
It holds that $\OptValFunc\eqref{EqJeroslowInequalityOptimization} = \OptValFunc\eqref{EqBLPWithLinking}$.
\end{lemma}

\begin{proof}
Recall that, for any $\SATReformulationBinaryVariable{} \in \{0, 1\}^n$, there exists unique $(\SATReformulationAuxVariable{}, \SATReformulationOutputVariable{})$ such that $(\SATReformulationBinaryVariable{}, \SATReformulationAuxVariable{}, \SATReformulationOutputVariable{}) \in \FeasSolSet\eqref{EqJeroslowInequalityOptimization}$, and
in light of Lemma~\ref{lemma:LexLowerLevelSolution}, there exists $\LexUVar, \LexLVarS, \LexLVarT, \LexLVarU, \LexLVarV, \LexLVarW$ such that $(\LexUVar, \LexLVarS, \LexLVarT, \LexLVarU, \LexLVarV, \LexLVarW, \SATReformulationBinaryVariable{}, \SATReformulationAuxVariable{}, \SATReformulationOutputVariable{}) \in \FeasSolSet\eqref{EqBLPWithLinking}$.
The objective functions of the two problems are the same, and thus $\OptValFunc\eqref{EqJeroslowInequalityOptimization} \le \OptValFunc\eqref{EqBLPWithLinking}$.

Now, pick any $(\LexUVar, \LexLVarS, \LexLVarT, \LexLVarU, \LexLVarV, \LexLVarW, \SATReformulationBinaryVariable{}, \SATReformulationAuxVariable{}, \SATReformulationOutputVariable{}) \in \FeasSolSet\eqref{EqBLPWithLinking}$.
In particular, it satisfies the upper-level constraint $\LexLVarV \ge (1/2) \cdot \OneVector$, implying $\SATReformulationBinaryVariable{} \in \{0, 1\}^{n}$ (Lemma~\ref{lemma:LexLowerLevelSolution}).
Thus $\OptValFunc\eqref{EqJeroslowInequalityOptimization} \ge \OptValFunc\eqref{EqBLPWithLinking}$.
\MyQED
\end{proof}

Next, we remove $\LexLVarV \ge (1/2) \cdot \OneVector$ from the upper-level problem using Lemma~\ref{lemma:PenaltyReformulation}:
\begin{align}
\MyProblemTag{EqHardBLP}{(\SATBooleanFormula)}
\max_{\substack{\LexUVar, \LexLVarS, \LexLVarT, \LexLVarU, \LexLVarV,\\\LexLVarW, \SATReformulationBinaryVariable{}, \SATReformulationAuxVariable{}, \SATReformulationOutputVariable{}, \ReformPen}}
\{
\SATReformulationOutputVariable{} - \ReductionPenWeight \ReformPen
:
(\LexLVarS, \LexLVarT, \LexLVarU, \LexLVarV, \LexLVarW, x, y, z, \ReformPen)
\in
\OptSolSet\eqref{EqHardBLPLowerlevel}
\},
\end{align}
where
\begin{align}
\MyProblemTag{EqHardBLPLowerlevel}{(\LexUVar, \SATBooleanFormula)}
\min_{\substack{\LexLVarS, \LexLVarT, \LexLVarU, \LexLVarV,\\\LexLVarW, \SATReformulationBinaryVariable{}, \SATReformulationAuxVariable{}, \SATReformulationOutputVariable{}, \ReformPen}}
\left\{
\sum_{i = 1}^n 
\ReductionLexWeight^{n - i + 1} (\LexLVarS_{i} + \LexLVarT_{i} + \LexLVarU_{i})
+
\OneVector^{\top} \LexLVarV
:
\begin{array}{l}
\begin{array}[t]{l}
(\LexLVarS, \LexLVarT, \LexLVarU, \LexLVarV, \LexLVarW)
\in \FeasSolSet\eqref{EqSATBilevelLexProgrammingLowerLevel},
\\
A_{\SATBooleanFormula} \SATReformulationBinaryVariable{}
+ B_{\SATBooleanFormula} \SATReformulationAuxVariable{}
+ C_{\SATBooleanFormula} \SATReformulationOutputVariable{}
\ge d_{\SATBooleanFormula},
\\
\SATReformulationBinaryVariable{} = \LexLVarW,
\\
\LexLVarV + \ReformPen \OneVector \ge (1/2) \cdot \OneVector,
\\
\ZeroVector \le \SATReformulationBinaryVariable{}, \SATReformulationAuxVariable{}, \SATReformulationOutputVariable{}, \ReformPen \le \OneVector
\end{array}
\end{array}
\right\}
\end{align}
and $\ReductionPenWeight$ is the constant prescribed in Lemma~\ref{lemma:PenaltyReformulation}.
It holds that $\OptValFunc\eqref{EqBLPWithLinking} = \OptValFunc\eqref{EqHardBLP}$.
Note that the \EncodingSize{} of~\eqref{EqHardBLP} is polynomial in the encoding size of $\SATBooleanFormula$.
Theorem~\ref{theorem:ValueOfBLPWithSingleUpperLevelVariableIsNPHard} follows immediately from the above argument.

\begin{proof}[Theorem~\ref{theorem:ValueOfBLPWithSingleUpperLevelVariableIsNPHard}]
\phantomsection\label{Rev:LogSpaceReduction}
Membership follows from~\cite{Buchheim2023}, while hardness follows from~\eqref{EqHardBLP}.
\MyQED
\end{proof}

A few remarks are in order.
The optimal objective value of~\eqref{EqHardBLP} is $1$ if $\SATBooleanFormula$ is satisfiable and $0$ otherwise.
Thus, even approximate versions are $\ClassNP$-hard (see Corollary 4.6 in Jeroslow~\cite{Jeroslow1985}).
Second, instance~\eqref{EqHardBLP} contains constants of exponential magnitude, $\ReductionLexWeight^{i}$ for $i \in [\NVariables]$ and $\ReductionPenWeight$.  
However, such constants can be removed with auxiliary variables (Remark~\ref{Remark:PolyMagnitudeWeights}).
\phantomsection\label{Rev:Pessimistic} Third, our proof of \ClassNP{}-hardness for the decision version of the optimistic formulation of bilevel LP also carries over to the pessimistic formulation. Recall that in the pessimistic formulation the follower, when faced with multiple optimal solutions, selects the one least favorable to the leader. 
One can consider the modified lower-level problem
\begin{align*}
\lexmin_{\substack{\LexLVarS, \LexLVarT, \LexLVarU, \LexLVarV,\\\LexLVarW, \SATReformulationBinaryVariable{}, \SATReformulationAuxVariable{}, \SATReformulationOutputVariable{}, \ReformPen}}
\left\{
\begin{array}{r}
\displaystyle
\left(
\sum_{i = 1}^n 
\ReductionLexWeight^{n - i + 1} (\LexLVarS_{i} + \LexLVarT_{i} + \LexLVarU_{i})
+
\OneVector^{\top} \LexLVarV,
-\SATReformulationOutputVariable{} + \ReductionPenWeight \ReformPen
\right)
\qquad\qquad
\\
:
(\LexLVarS, \LexLVarT, \LexLVarU, \LexLVarV,\LexLVarW, \SATReformulationBinaryVariable{}, \SATReformulationAuxVariable{}, \SATReformulationOutputVariable{}, \ReformPen) \in 
\FeasSolSet\eqref{EqHardBLPLowerlevel}
\end{array}
\right\}.
\end{align*}
This can be rewritten as a single-objective LP of polynomial size.
The resulting bilevel LP has the identical optimistic and pessimistic formulations since the upper-level objective value is constant over the optimal solution set of the lower-level problem. 
This modification also preserves the optimal objective value, implying the hardness of the pessimistic variant.

\section{Complexity of Local Search}
\label{sec:ComplexityOfVerifyingASolution}

To further understand the computational complexity of \eqref{EqNotationBLP}, in this section, we study the complexity of searching for a solution.
We begin by discussing the \ClassCoNP{}-completeness of deciding whether a given point is a globally optimal solution.
We then turn our attention to the problem of finding a locally optimal solution.
We show that it is possible to compute a locally optimal solution in polynomial time.

Consider the following decision problem:
\begin{align*}
&
\parbox{0.75\textwidth}{Given a rational instance~\eqref{EqNotationBLP} and a rational point $(x_1, x_2)$, is it an optimal solution?}
\label{GlobalOptimalProblemAnchor}
\tag{\GlobalOptimalProblemName{}}
\end{align*}

By modifying instance~\eqref{EqHardBLP}, we obtain the following result.
\begin{proposition}
\label{prop:ComplexityOfVerifyingOptimalSolution}
The decision problem \GlobalOptimalProblem{} is \ClassCoNP{}-complete.
\end{proposition}

\begin{proof}
We show the complement \CompGlobalOptimalProblem{} is \ClassNP{}-complete.
The membership of \CompGlobalOptimalProblem{} in \ClassNP{} is similar to the proof of the membership of the decision version of bilevel LP in \ClassNP{}~\cite{Buchheim2023}.
Thus, we only show \CompGlobalOptimalProblem{} is \ClassNP{}-hard, by reducing \DecisionProblemSAT{} to \CompGlobalOptimalProblem{}.
Let $\SATBooleanFormula$ be a Boolean Formula, and consider
\begin{align}
\MyProblemTag{EqHardnessOfRecognitionOfSolution}{(\SATBooleanFormula)}
\max_{\substack{\LexUVar, \LexLVarS, \LexLVarT, \LexLVarU, \LexLVarV,\\\LexLVarW, \SATReformulationBinaryVariable{}, \SATReformulationAuxVariable{}, \SATReformulationOutputVariable{}}}
\left\{
\SATReformulationOutputVariable{}
:
\LexLVarV \ge (1/2) \OneVector,
(\LexLVarS, \LexLVarT, \LexLVarU, \LexLVarV, \LexLVarW, \SATReformulationBinaryVariable{}, \SATReformulationAuxVariable{}, \SATReformulationOutputVariable{})
\in
\OptSolSet\eqref{EqHardnessOfRecognitionOfSolutionLowerLevel}
\right\}
\end{align}
where
\begin{align}
\MyProblemTag{EqHardnessOfRecognitionOfSolutionLowerLevel}{(\LexUVar, \SATBooleanFormula)}
\min_{\substack{\LexLVarS, \LexLVarT, \LexLVarU, \LexLVarV,\\\LexLVarW, \SATReformulationBinaryVariable{}, \SATReformulationAuxVariable{}, \SATReformulationOutputVariable{}}}
\left\{
\displaystyle
\sum_{i = 1}^{n + 1}
\ReductionLexWeight^{n - i + 2} (\LexLVarS_{i} + \LexLVarT_{i} + \LexLVarU_{i})
+
\OneVector^{\top} \LexLVarV
:
\begin{array}{l}
(\LexLVarS, \LexLVarT, \LexLVarU, \LexLVarV, \LexLVarW)
\in \FeasSolSet\MyEqref{EqSATBilevelLexProgrammingLowerLevel}{(\NVariables + 1, \LexUVar)},
\\
\SATReformulationBinaryVariable{}_i = \LexLVarW_i,
\ \ 
\forall \ForAllInRange{i}{\NVariables},
\\
\SATReformulationOutputVariable{} \le \LexLVarW_{\NVariables + 1},
\\
A_{\SATBooleanFormula} \SATReformulationBinaryVariable{}
+ B_{\SATBooleanFormula} \SATReformulationAuxVariable{}
+ C_{\SATBooleanFormula} \SATReformulationOutputVariable{}
\ge d_{\SATBooleanFormula} \LexLVarW_{\NVariables + 1},
\\
\ZeroVector \le \SATReformulationBinaryVariable{}, \SATReformulationAuxVariable{}, \SATReformulationOutputVariable{} \le \OneVector
\end{array}
\right\}
\end{align}
By Lemma~\ref{lemma:LexLowerLevelSolution}, $\LexUVar = \LexLVarS = \LexLVarT = \LexLVarU = x = y = z = \ZeroVector$, $\LexLVarV = (1/2) \cdot \OneVector$ is feasible. 
It is not optimal if and only if  $\SATBooleanFormula$ is satisfiable.
One can construct an equivalent instance without upper-level constraints (Lemma~\ref{lemma:PenaltyReformulation}).
Thus, \CompGlobalOptimalProblem{} is \ClassNP{}-hard.
\MyQED
\end{proof}

Proposition~\ref{prop:ComplexityOfVerifyingOptimalSolution} shows that merely deciding whether a given point is an optimal solution is \ClassCoNP{}-complete.
Therefore, searching for a globally optimal solution is likely to be computationally intractable.
A natural question is whether replacing global optimality with local optimality leads to a tractable problem.

We say that $(\bar{x}_1, \bar{x}_2) \in \mathbb{R}^{\NVariables + 1}$ is a \emph{locally optimal solution} to~\eqref{EqNotationBLP} if $(\bar{x}_1, \bar{x}_2) \in \FeasSolSet\eqref{EqNotationBLP}$ and there exists $\epsilon > 0$ such that
$$
c_{1 1} \bar{x}_1 + c_{1 2}^{\top} \bar{x}_2 \le c_{1 1} x_1 + c_{1 2}^{\top} x_2,
\quad \forall\, (x_1, x_2) \in \FeasSolSet\eqref{EqNotationBLP} \cap \mathcal{B}_{\epsilon}(\bar{x}_1, \bar{x}_2),
$$
where
$
\mathcal{B}_{\epsilon}(\bar{x}_1, \bar{x}_2) = \{ (x_1, x_2) : \| (x_1, x_2) - (\bar{x}_1, \bar{x}_2) \|_{\infty} \le \epsilon \}.
$
The following result holds.
\begin{theorem}
\label{theorem:ComplexityOfLocalSearch}
There exists a polynomial-time algorithm that finds a locally optimal solution to a rational instance of~\eqref{EqNotationBLP}.
\end{theorem}

Vicente \EtAl{}~\cite{vicente1994descent} show that recognizing local optimality is \ClassNP{}-hard for general bilevel LPs. In contrast, the proof of Theorem~\ref{theorem:ComplexityOfLocalSearch} shows that the recognition can be done efficiently for~\eqref{EqNotationBLP}, suggesting a fundamental difference between general bilevel LPs and \eqref{EqNotationBLP}.
Below, we show that if there is a locally optimal solution, there is a rational one of polynomial size.
We use a binary search to find locally optimal $x_1$.
The binary search argument closely follows that of \cite{kwek2003optimal}.

In this context, it is handy to use the following optimization:
\begin{align}
\MyProblemTag{ReformulationLocalOPT}{}
\min_{x_1} \left\{ \PartialOptValFunc(x_1) : 0\leq x_1 \leq 1 \right\}, 
\end{align}
where
\begin{align*}
\PartialOptValFunc(x_1) :=
\min_{x_2}
\left\{
c_{1 1} x_1 + c_{1 2}^{\top} x_2
:
x_2 \in \Argmin_{x_2'} 
\left\{ c_{2 2}^{\top} x_2' 
: 
\begin{array}{l}
A_{2 1} x_1 + A_{2 2} x_2' \ge b_2,
\\
0 \le x_2' \le \OneVector
\end{array}
\right\} 
\right\}.
\end{align*}
We note that given any locally optimal solution to~\eqref{ReformulationLocalOPT}, one can compute a locally optimal solution to~\eqref{EqNotationBLP} efficiently.
Given $x_1$ locally optimal for~\eqref{ReformulationLocalOPT}, there exists $x_2$ such that $(x_1, x_2) \in \FeasSolSet\eqref{EqNotationBLP}$ and $c_{1 1} x_1 + c_{1 2}^{\top} x_2 = \PartialOptValFunc(x_1)$.
Then, $(x_1, x_2)$ is locally optimal for \eqref{EqNotationBLP} since we have $c_{1 1} x_1' + c_{1 2}^{\top} x_2' \ge \PartialOptValFunc(x_1')$ for any $(x_1', x_2') \in \FeasSolSet\eqref{EqNotationBLP}$.
Moreover, evaluating $\PartialOptValFunc(x_1)$ is computationally efficient, as it amounts to a solution of one lexicographic LP:
\begin{align*}
\lexmin_{x_2}
\left\{
( 
c_{2 2}^{\top} x_2, c_{1 1} x_1 + c_{1 2}^{\top} x_2 
)
:
\begin{array}{l}
A_{2 1} x_1 + A_{2 2} x_2 \ge b_2,
\\
0 \le x_2 \le \OneVector
\end{array}
\right\}.
\end{align*}
Let
$
l 
= 
\min_{x_1, x_2} \{ 
x_1 
: 
(x_1, x_2) \in 
\FeasSolSet\eqref{EqNotationBLP}
\},
u 
= 
\max_{x_1, x_2} \{ 
x_1 
: 
(x_1, x_2) \in 
\FeasSolSet\eqref{EqNotationBLP}
\}.
$
Either $l = \infty$ and $u = -\infty$, or $l$ and $u$ are rationals of polynomial size (Lemma~\ref{lemma:estimate-of-primal-linear-inequality}).
\begin{lemma}
\label{lemma:PiecewiseLinearityOfPartialValueFunction}
Suppose~$l$ and $u$ are finite.
Then, $\PartialOptValFunc$ is continuous and piecewise-linear over $[l, u]$.
Furthermore, all kinks of $\PartialOptValFunc$ are rationals of polynomial size.
\end{lemma}

%

This follows from a standard argument of optimal basic solutions in lexicographic LP~\cite{isermann1982linear} and hence omitted (see also~\cite{Sugishita2026}).
By Lemma~\ref{lemma:PiecewiseLinearityOfPartialValueFunction} one only needs to search for a locally optimal kink of $\PartialOptValFunc$, for which we need one-sided derivatives.

\begin{lemma}
\label{lemma:OneSidedDerivativeOfPartialOptValFunction}
Given a rational point $x_1 \in \mathbb{Q}$, there is a polynomial-time algorithm that computes the one-sided derivatives of $\PartialOptValFunc$ at $x_1$, whenever they exist.
\end{lemma}

To prove this, we use the continued fraction method.

\begin{lemma}[Continued fraction method~\cite{Schrijver1998}]
\label{lemma:ContinuedFractionMethod}
There exists a polynomial algorithm which, for given rational number $\alpha$ and natural number $M$, tests if there exists a rational number $p/q$ with $1 \le q \le M$ and $|\alpha - p/q| < 1/2M^2$, and if so, finds this (unique) rational number.
\end{lemma}

\begin{proof}[Lemma~\ref{lemma:OneSidedDerivativeOfPartialOptValFunction}]
We only consider the left derivative, for the right derivative can be computed analogously.
Assume $x_1 \in (l, u]$, as otherwise the left derivative does not exist.
Let $s$ be an integer bounding the size of kinks of $\PartialOptValFunc$ (including $l, u$).
Search for rational $r = p/q$ such that $1 \le q \le 2^{s}$, $| x_1 - r | < 1/ 2^{2 s + 1}$ (Lemma~\ref{lemma:ContinuedFractionMethod}).
If such $r$ exists and if $r < x_1$, let $r' = r$.
Otherwise, let $r' = x_1 - 1/2^{2 s + 1}$ (note that $r' > l$).
There is no rational of size at most $s$ in $(r', x_1)$.
Thus, $\PartialOptValFunc$ is linear over $[r', x_1]$. 
Compute the slope of this segment and return it.
\MyQED
\end{proof}


We are now ready to prove Theorem~\ref{theorem:ComplexityOfLocalSearch}.

\begin{proof}[Theorem~\ref{theorem:ComplexityOfLocalSearch}]
We describe a procedure for finding a locally optimal solution to~\eqref{ReformulationLocalOPT}, from which one can compute a locally optimal solution to~\eqref{EqNotationBLP}.
If $l$ and $u$ are not finite, \eqref{EqNotationBLP} is infeasible and has no locally optimal solutions.
If $l = u$, then $x_1 = l$ is locally optimal.
Thus, we assume $0 \le l < u \le 1$.
Check whether either $l$ or $u$ is locally optimal by evaluating the one-sided derivatives.
If so, return the corresponding point.
Otherwise, we must have that the right derivative of $\PartialOptValFunc$ at $l$ is negative, and the left derivative at $u$ is positive.
Let $v := (l + u)/2$ and evaluate the one-sided derivatives of $\PartialOptValFunc$ at $v$.
If $v$ is locally optimal, return it.
If the left derivative at $v$ is positive, then by the extreme value theorem, a locally optimal solution (in particular, a locally optimal kink) must exist in $[l, v]$, so let $u := v$.
Otherwise, there is a locally optimal solution in $[v, u]$, thus update $l := v$.
Repeat until a locally optimal solution is found, or $u - l \le 1 / 2^{2s + 1}$, where $s$ is an upper bound on the size of the kinks of $\PartialOptValFunc$.
At this point, $[l, u]$ contains at most one rational number of size at most $s$, which must be a locally optimal solution.
Compute it with the continued fraction method (Lemma~\ref{lemma:ContinuedFractionMethod}).
\MyQED
\end{proof}

\section{Extension to Bilevel Pricing Problem}
\label{sec:Pricing}

In this section, extending the technique developed in the earlier sections, we study the bilevel pricing problem whose lower-level player solves an LP, which we refer to as the pricing bilevel LP:
\begin{equation}
\MyProblemTag{EqPricingBLPGeneral}{}
\min_{x_1, x_2} 
\left\{ 
\begin{array}{l}
c_{1 1}^{\top} x_1 + c_{1 2}^{\top} x_2 + x_1^{\top} D_1 x_2 
\\
\qquad\qquad
: 
\begin{array}{l}
A_{1 1} x_1 + A_{1 2} x_2 \ge b_1, \\
\displaystyle
x_2 \in \arg\min_{x_2'} \{ c_{2 2}^{\top} x_2' + x_1^{\top} D_2 x_2' : A_{2 1} x_ 1 + A_{2 2} x_2' \ge b_2 \} 
\end{array}
\end{array}
\right\},
\end{equation}
where $A_{1 1} \in \mathbb{Q}^{\NConstraints_1 \times \NVariables_1}, A_{1 2} \in \mathbb{Q}^{\NConstraints_1 \times \NVariables_2}, A_{2 1} \in \mathbb{Q}^{\NConstraints_2 \times \NVariables_1}, A_{2 2} \in \mathbb{Q}^{\NConstraints_2 \times \NVariables_2}$, $b_1 \in \mathbb{Q}^{\NConstraints_1}, b_2 \in \mathbb{Q}^{\NConstraints_2}$, $c_{1 1} \in \mathbb{Q}^{\NVariables_1}$, $D_1, D_2 \in \mathbb{Q}^{\NVariables_1 \times \NVariables_2}$ and $c_{1 2}, c_{2 2} \in \mathbb{Q}^{\NVariables_2}$.
As in the proof of Lemma~\ref{lemma:SolutionSizeEstimate}, $\FeasSolSet\eqref{EqPricingBLPGeneral}$ can be written as
\begin{align}
\notag
\FeasSolSet\eqref{EqPricingBLPGeneral}
&=
\left\{
\begin{pmatrix}
x_1 \\ x_2
\end{pmatrix}
:
\exists \lambda \in \NonNegativeReals^{\NConstraints_2}
\text{ s.t. }
\begin{array}{l}
A_{1 1} x_1 + A_{1 2} x_2 \ge b_1, \\
A_{2 1} x_1 + A_{2 2} x_2 \ge b_2, \\
\lambda^{\top} ( A_{2 1} x_1 + A_{2 2} x_2 - b_2) = 0, \\
A_{2 2}^{\top} \lambda = c_{2 2} + D_2^{\top} x_1
\end{array}
\right\},
\end{align}
or
\begin{align}
\notag
\FeasSolSet\eqref{EqPricingBLPGeneral}
&=
\bigcup_{\omega \in \{0, 1\}^{\NConstraints_2}}
\underbrace{
\left\{
\begin{pmatrix}
x_1 \\ x_2
\end{pmatrix}
:
\exists \lambda \in \NonNegativeReals^{\NConstraints_2}
\text{ s.t. }
\begin{array}{ll}
A_{1 1} x_1 + A_{1 2} x_2 \ge b_1, \\
A_{2 1} x_1 + A_{2 2} x_2 \ge b_2, \\
\omega^{\top} (A_{2 1} x_1 + A_{2 2} x_2 - b_2) = 0, \\
(1 - \omega)^{\top} \lambda = 0, \\
A_{2 2}^{\top} \lambda = c_{2 2} + D_2^{\top} x_1
\end{array}
\right\}}_{:= \bar{U}_\omega}.
\end{align}



Thus, we have the following.
\begin{theorem}
\label{theorem:ComplexityOfGeneralPricingBLP}
\mbox{}
\begin{enumerate}
\item
If 
$\OptSolSet\eqref{EqPricingBLPGeneral} \not= \emptyset$ ($\FeasSolSet\eqref{EqPricingBLPGeneral} \ne \emptyset$, respectively),
then there exists a rational optimal (feasible, respectively) solution of polynomial encoding size.
\item \label{ii_theorem3}
Given~\eqref{EqPricingBLPGeneral} and a rational number $t$, deciding if $\OptValFunc\eqref{EqPricingBLPGeneral} \le t$ is NP-complete.
\end{enumerate}
\end{theorem}

\begin{proof}
\begin{enumerate}
\item 
We only show the result for an optimal solution; the other one can be shown similarly.
Pick $\omega \in \{0, 1\}^{\NConstraints_2}$ such that $\OptSolSet\eqref{EqEqPricingBLPSolutionSizeEstimateDecomposed} \subseteq \OptSolSet\eqref{EqPricingBLPGeneral}$, where
\begin{equation}
\MyProblemTag{EqEqPricingBLPSolutionSizeEstimateDecomposed}{(\omega)}
\min_{x_1, x_2} \{
c_{1 1}^{\top} x_1 + c_{1 2}^{\top} x_2 + x_1^{\top} D_1 x_2 : (x_1, x_2) \in \bar{U}_{\omega}
\}.
\end{equation}
The claim follows from the existence of a rational $(x_1, x_2) \in \OptSolSet\eqref{EqEqPricingBLPSolutionSizeEstimateDecomposed}$~\cite{vavasis1990quadratic}.
\item 
Hardness follows from Theorem~\ref{theorem:ValueOfBLPWithSingleUpperLevelVariableIsNPHard}.
For membership, observe that, given a threshold $t$, $\OptValFunc\eqref{EqPricingBLPGeneral} \le t$ if and only if there exists $\omega \in \{0, 1\}^{\NConstraints_2}$ such that $\OptValFunc\eqref{EqEqPricingBLPSolutionSizeEstimateDecomposed} \le t$.
Thus, an NP-machine can guess $\omega$ together with the certificate of $\OptValFunc\eqref{EqEqPricingBLPSolutionSizeEstimateDecomposed} \le t$ as discussed in~\cite{vavasis1990quadratic}.
\MyQED
\end{enumerate}
\end{proof}

We now show that statement~\ref{ii_theorem3} of Theorem~\ref{theorem:ComplexityOfGeneralPricingBLP} holds under a far more restrictive structure, in which the upper-level variable is a single price that appears only in the follower's objective:
\begin{equation}
\MyProblemTag{EqPricingBLP}{}
\min_{x_1 \in [0, 1], x_2} 
\left\{ 
c_{1 2}^{\top} x_2
: 
x_2 \in \arg\min_{x_2'} \{ c_{2 2}^{\top} x_2' + x_1 D_2 x_2' : A_{2 2} x_2' \ge b_2 \} 
\right\}.
\end{equation}
We will examine the hardness of the decision version of \eqref{EqPricingBLP}.

Recall that~\eqref{EqSATBilevelLexProgrammingLowerLevel}, or~\eqref{EqSATBilevelWeightedProblemLowerLevelTwo}, played a key role in the analysis of \eqref{EqNotationBLP}.
In the dual of~\eqref{EqSATBilevelWeightedProblemLowerLevelTwo}, $\SecPricingLexUVar$ appears in a cost coefficient, and it is of the form of the lower-level problem of~\eqref{EqPricingBLP}.
To analyze~\eqref{EqPricingBLP}, we use a slightly modified variant of~\eqref{EqSATBilevelLexProgrammingLowerLevel}.
Let $n \in \PositiveIntegers$ and
\begin{alignat}{3}
\MyProblemTag{EqPrimalLex}{(\NVariables, \SecPricingLexUVar)}
\lexmin_{\SecPricingLexLVarR{}{}, \SecPricingLexLVarS, \SecPricingLexLVarT, \SecPricingLexLVarU}
\left\{
\begin{array}{l}
\displaystyle
\left( \sum_{j = 1}^3 \SecPricingLexLVarR{1}{j} + \SecPricingLexLVarS_1 + \SecPricingLexLVarT_1 + \SecPricingLexLVarU_1, \ldots, \sum_{j = 1}^3 \SecPricingLexLVarR{n}{j} + \SecPricingLexLVarS_n + \SecPricingLexLVarT_n + \SecPricingLexLVarU_n \right)
\\
\qquad
\qquad
:
\begin{array}{llllll}
&
\text{\makebox[0cm][l]{$
\SecPricingLexLVarU_{1} \ge \SecPricingLexUVar,
\
\SecPricingLexLVarS \ge (3/2) \SecPricingLexLVarU - (1/2) \cdot \OneVector,
\
\SecPricingLexLVarT \ge (3/2) \SecPricingLexLVarU - \OneVector,
$}}
\\
&
\SecPricingLexLVarR{i}{j} \ge (3/2) \SecPricingLexLVarU_i - \SecPricingG{j},
&&
\qquad
\forall \ForAllInRange{i}{n}, \ForAllInRange{j}{3},
\quad
\\
&
\SecPricingLexLVarU_{i + 1} \ge 3 \SecPricingLexLVarU_i - 4 (\LexLVarS_i - \LexLVarT_i),
&&
\qquad
\forall \ForAllInRange{i}{\NVariables - 1},
\quad
\\
&
\ZeroVector \le \LexLVarS, \LexLVarT, \LexLVarU,  \SecPricingLexLVarR{}{}
\end{array}
\end{array}
\right\},
\end{alignat}
where $\SecPricingG{} = (1/2, 3/4, 1)$.
Instance~\eqref{EqPrimalLex} is obtained by removing $\LexLVarW$ and $\LexLVarV$, relaxing variable upper bounds in~\eqref{EqSATBilevelLexProgrammingLowerLevel} and adding $\SecPricingLexLVarR{}{}$ together with its constraints.
Let $\SecPricingLexWeight$ be sufficiently large (Lemma~\ref{Lemma:LexicographicToWeighted}) so that $\SecPricingLexWeight \ge 40$ and $\OptSolSet\eqref{EqPrimalLex} = \OptSolSet\eqref{EqPrimalAggregated}$, where
\begin{alignat}{3}
\MyProblemTag{EqPrimalAggregated}{(\NVariables, \SecPricingLexUVar)}
\min_{\SecPricingLexLVarR{}{}, \SecPricingLexLVarS, \SecPricingLexLVarT, \SecPricingLexLVarU}
\left\{
\sum_{i = 1}^{n} \SecPricingLexWeight^{n - i} \left(\sum_{j = 1}^3 \SecPricingLexLVarR{i}{j} + \SecPricingLexLVarS_i + \SecPricingLexLVarT_i + \SecPricingLexLVarU_i\right)
:
(\SecPricingLexLVarR{}{}, \SecPricingLexLVarS, \SecPricingLexLVarT, \SecPricingLexLVarU) \in \FeasSolSet\eqref{EqPrimalLex}
\right\}.
\end{alignat}
For any $\SecPricingLexUVar \in [0, 1]$,~\eqref{EqPrimalAggregated} has a unique optimal solution, as shown in Lemma~\ref{lemma:SecPricing:LexLowerLevelSolution}.
The proof is similar to that of Lemma~\ref{lemma:LexLowerLevelSolution} and hence omitted.
\begin{lemma}
\label{lemma:SecPricing:LexLowerLevelSolution}
For any $\SecPricingLexUVar \in [0, 1]$, it holds that
\begin{align*}
&
\OptSolSet\eqref{EqPrimalAggregated} =
\\
&
\left\{
\begin{pmatrix}
\SecPricingLexLVarR{}{} \\ \SecPricingLexLVarS \\ \SecPricingLexLVarT \\ \SecPricingLexLVarU
\end{pmatrix}
:
\begin{array}{lll}
\SecPricingLexLVarR{i}{j} = [(3/2) \SecPricingLexLVarU_i - \SecPricingG{j}]^+, 
\quad
&
\SecPricingLexLVarS_i = [(3/2) \SecPricingLexLVarU_i - 1/2]^+, 
\quad
&
\forall \ForAllInRange{i}{\NVariables}, \ForAllInRange{j}{3},
\\
\SecPricingLexLVarT_i = [(3/2) \SecPricingLexLVarU_i - 1]^+, 
\quad
&
\SecPricingLexLVarU_i = f_u^{i - 1}(\SecPricingLexUVar),
\quad
&
\forall \ForAllInRange{i}{\NVariables}
\end{array}
\right\},
\end{align*}
where $f_u^{i}$ is the composition of $f_u$ for $i$ times.
\end{lemma}
We will write $\SecPricingUMap(\SecPricingLexUVar)$ be the vector whose $i$-th element is given by $f_u^{i - 1}(\SecPricingLexUVar)$.

The dual problem for~\eqref{EqPrimalAggregated} is given by
\begin{alignat}{3}
&
\MyProblemTag{EqDual}{(\NVariables, \SecPricingLexUVar)}
\max_{\SecPricingDualVarOne, \SecPricingDualVarTwo, \SecPricingDualVarThree, \SecPricingDualVarFour, \SecPricingDualVarFive}
\left\{
\begin{array}{l}
\displaystyle
\SecPricingLexUVar \SecPricingDualVarOne
- \frac12 \OneVector^{\top} \SecPricingDualVarTwo
- \OneVector^{\top} \SecPricingDualVarThree
- \sum_{i = 1}^n \sum_{j = 1}^{3} \SecPricingG{j} \SecPricingDualVarFour_{i j}
\\
\qquad
:
\begin{array}{lll}
\SecPricingDualVarFour_{i j} \le \SecPricingLexWeight^{n - i},
&
&
\forall \ForAllInRange{i}{\NVariables}, \ForAllInRange{j}{3},
\\
\SecPricingDualVarTwo_{i} + 4 \SecPricingDualVarFive_{i} \le \SecPricingLexWeight^{n - i},
&
\SecPricingDualVarThree_{i} - 4 \SecPricingDualVarFive_{i} \le \SecPricingLexWeight^{n - i},
\quad\quad\quad\quad\ \
&
\forall \ForAllInRange{i}{\NVariables},
\\
\text{\makebox[0cm][l]{$
\SecPricingDualVarOne - (3/2) (\SecPricingDualVarTwo_1 + \SecPricingDualVarThree_1 + \sum_{j = 1}^3 \SecPricingDualVarFour_{1 j}) - 3 \SecPricingDualVarFive_
{1} \le \SecPricingLexWeight^{n - 1},
$}}
\\
\text{\makebox[0cm][l]{$
\SecPricingDualVarFive_
{i} -(3/2) (\SecPricingDualVarTwo_{i + 1} + \SecPricingDualVarThree_{i + 1} + \sum_{j = 1}^3 \SecPricingDualVarFour_{i + 1 \, j}) - 3 \SecPricingDualVarFive_
{i + 1} \le \SecPricingLexWeight^{n - i - 1},
$}}
\\
&
&
\forall \ForAllInRange{i}{\NVariables - 1},
\\
\text{\makebox[0cm][l]{$
\SecPricingDualVarOne \in \NonNegativeReals, \SecPricingDualVarTwo \in \NonNegativeReals^{n}, \SecPricingDualVarThree \in \NonNegativeReals^n, \SecPricingDualVarFour \in \NonNegativeReals^{\NVariables \times 3}, \SecPricingDualVarFive \in \NonNegativeReals^{\NVariables}, \SecPricingDualVarFive_{n} = 0
$}}
\end{array}
\end{array}
\right\}.
\end{alignat}

We first show that $\FeasSolSet\eqref{EqDual}$ is bounded.

\begin{lemma}
\label{lemma:SecPricing:BoundedLowerLevelFeasibleSet}
It holds that $(\SecPricingDualVarOne, \SecPricingDualVarTwo, \SecPricingDualVarThree, \SecPricingDualVarFour, \SecPricingDualVarFive) \in \FeasSolSet\eqref{EqDual}$ implies $0 \le \SecPricingDualVarOne, \SecPricingDualVarTwo_i, \SecPricingDualVarThree_i, \SecPricingDualVarFour_{i j}, \SecPricingDualVarFive_i \le 10 \SecPricingLexWeight^{n - 1}$ for any $i$ and $j$.
\end{lemma}

\begin{proof}
We have $4 \SecPricingDualVarFive_i \le \SecPricingLexWeight^{n - i} - \SecPricingDualVarTwo_i \le \SecPricingLexWeight^{n - i}$, and thus $\SecPricingDualVarFive_i \le \SecPricingLexWeight^{n - i} / 4$.
Furthermore, by combining constraints, $\SecPricingDualVarTwo_i + \SecPricingDualVarThree_i \le 2 \SecPricingLexWeight^{n - i}$.
Together with $\SecPricingDualVarFour_{i j} \le \SecPricingLexWeight^{n - i}$, $j = 1, 2, 3$, we get $\SecPricingDualVarOne \le \SecPricingLexWeight^{n - 1} + (3/2) 5 \SecPricingLexWeight^{n - 1} + (3/4) \SecPricingLexWeight^{n - 1} < 10 \SecPricingLexWeight^{n - 1}$.
\MyQED
\end{proof}

Next, we study properties of optimal solutions for~\eqref{EqDual}.
The optimality condition based on the complementary slackness \cite{sierksma2015linear} (Theorem 4.3.1) yields that, given $(\SecPricingLexLVarR{}{}, \SecPricingLexLVarS, \SecPricingLexLVarT, \SecPricingLexLVarU) \in \FeasSolSet\eqref{EqPrimalAggregated}$ and $(\SecPricingDualVarOne, \SecPricingDualVarTwo, \SecPricingDualVarThree, \SecPricingDualVarFour, \SecPricingDualVarFive) \in \FeasSolSet\eqref{EqDual}$, they are optimal if and only if
\begingroup  
\newcommand{\MyForAllText}[1]{\tag*{\makebox*{$\forall \ForAllInRange{i}{n}, \ForAllInRange{j}{3},$}[l]{$#1$}}}
\begin{align}
&
(\LexLVarU_{1} - \SecPricingLexUVar) \SecPricingDualVarOne = 0,
\notag
\\
&
(\SecPricingLexLVarS_i - (3/2) \SecPricingLexLVarU_i + (1/2)) \SecPricingDualVarTwo_i = 0,
\
(\SecPricingLexLVarT_i - (3/2) \SecPricingLexLVarU_i + 1) \SecPricingDualVarThree_i = 0
\MyForAllText{\forall \ForAllInRange{i}{n},}
\\
&
(\SecPricingLexLVarR{i}{j} - (3/2) \SecPricingLexLVarU_i + \SecPricingG{j}) \SecPricingDualVarFour_{i j} = 0,
\MyForAllText{\forall \ForAllInRange{i}{n}, \ForAllInRange{j}{3},}
\\
&
(\SecPricingLexLVarU_{i + 1} - 3 \SecPricingLexLVarU_i + 4 (\SecPricingLexLVarS_i - \SecPricingLexLVarT_i)) \SecPricingDualVarFive_i = 0,
\MyForAllText{\forall \ForAllInRange{i}{n - 1},}
\\
&
(\SecPricingLexWeight^{n - i} - \SecPricingDualVarFour_{i j}) \SecPricingLexLVarR{i}{j} = 0,
\MyForAllText{\forall \ForAllInRange{i}{n}, \ForAllInRange{j}{3},}
\\
&
(\SecPricingLexWeight^{n - i} - \SecPricingDualVarTwo_{i} - 4 \SecPricingDualVarFive_i) \SecPricingLexLVarS_i = 0,
\
(\SecPricingLexWeight^{n - i} - \SecPricingDualVarThree_{i} + 4 \SecPricingDualVarFive_i) \SecPricingLexLVarT_i = 0,
\MyForAllText{\forall \ForAllInRange{i}{n},}
\\
&
\textstyle
(\SecPricingLexWeight^{n - 1} - \SecPricingDualVarOne + (3/2) (\SecPricingDualVarTwo_1 + \SecPricingDualVarThree_1 + \sum_{j = 1}^3 \SecPricingDualVarFour_{1 j}) + 3 \SecPricingDualVarFive_{1}) \SecPricingLexLVarU_1 = 0,
\notag
\\
&
\textstyle
(\SecPricingLexWeight^{n - i - 1} - \SecPricingDualVarFive_i + (3/2) (\SecPricingDualVarTwo_{i + 1} + \SecPricingDualVarThree_{i + 1} + \sum_{j = 1}^3 \SecPricingDualVarFour_{i + 1 \, j}) + 3 \SecPricingDualVarFive_{i + 1}) \SecPricingLexLVarU_{i + 1} = 0,
\tag*{\makebox*{$\forall \ForAllInRange{i}{n}, \ForAllInRange{j}{3},$}[l]{$\forall \ForAllInRange{i}{n - 1},$}}
\end{align}
\endgroup  

If $\SecPricingUMap_i(\SecPricingLexUVar) < (2/3) \SecPricingG{j}$, constraint $\SecPricingLexLVarR{i}{j} \ge (3/2) \SecPricingLexLVarU_i - \SecPricingG{j}$ in~\eqref{EqPrimalAggregated} is inactive, and thus the corresponding dual variable $\SecPricingDualVarFour_{i j}$ must be zero.
Similarly, $\SecPricingUMap_i(\SecPricingLexUVar) > (2/3) \SecPricingG{j}$ implies $\SecPricingDualVarFour_{i j} = \SecPricingLexWeight^{n - i}$.
Formally, we have the following.

\begin{lemma}
\label{Lemma:DualSolutionAsFlag}
Let $\SecPricingLexUVar \in [0, 1]$ and
$(\SecPricingDualVarOne, \SecPricingDualVarTwo, \SecPricingDualVarThree, \SecPricingDualVarFour, \SecPricingDualVarFive) \in \mathcal{S}\eqref{EqDual}$.
Then, 
\begin{alignat*}{4}
\SecPricingUMap_i(\SecPricingLexUVar) < (2/3) \SecPricingG{j} \Longrightarrow \SecPricingDualVarFour_{i j} = 0, 
\ \
\SecPricingUMap_i(\SecPricingLexUVar) > (2/3) \SecPricingG{j} \Longrightarrow \SecPricingDualVarFour_{i j} = \SecPricingLexWeight^{n - i}, 
\ \
\forall \ForAllInRange{i}{n}, \ForAllInRange{j}{3}.
\end{alignat*}
\end{lemma}

\begin{proof}
We only show the second implication as the first one can be shown analogously.
Let $(\SecPricingLexLVarR{}{}, \SecPricingLexLVarS, \SecPricingLexLVarT, \SecPricingLexLVarU) \in \OptSolSet\eqref{EqPrimalAggregated}$.
By Lemma~\ref{lemma:SecPricing:LexLowerLevelSolution}, 
$
\SecPricingLexLVarR{i}{j} = [(3/2) \SecPricingUMap_i(\SecPricingLexUVar) - \SecPricingG{j}]^+.
$
By the complementary slackness condition, we have
$
(\SecPricingLexWeight^{n - i} - \SecPricingDualVarFour_{i j}) [(3/2) \SecPricingUMap_i(\SecPricingLexUVar) - \SecPricingG{j}]^+ = 0.
$
The claim follows immediately.
\MyQED
\end{proof}
When $\SecPricingUMap_i(\SecPricingLexUVar) = (2/3) \SecPricingG{j}$, any $\SecPricingDualVarFour_{i j} \in [0, \SecPricingLexWeight^{n - i}]$ is optimal.
Since this result is not strictly necessary in the following argument, we only provide a proof sketch.

\begin{lemma}
\label{lemma:SecPricing:LexLowerLevelSolutionOptVsPes}
Let $\SecPricingLexUVar \in [0, 1]$ and
$$
\SecPricingSetOfDeltas(\SecPricingLexUVar) :=
\left\{
\SecPricingDualVarFour
:
\begin{array}{ll}
\SecPricingDualVarFour_{i j} = 0, & \forall \ForAllInRange{i}{\NVariables}, \ForAllInRange{j}{3} \text{ s.t. } \SecPricingUMap_i(\SecPricingLexUVar) < (2/3) \SecPricingG{j}, \\
\SecPricingDualVarFour_{i j} \in [0, \SecPricingLexWeight^{\NVariables - i}], & \forall \ForAllInRange{i}{\NVariables}, \ForAllInRange{j}{3} \text{ s.t. } \SecPricingUMap_i(\SecPricingLexUVar) = (2/3) \SecPricingG{j}, \\
\SecPricingDualVarFour_{i j} = \SecPricingLexWeight^{\NVariables - i}, & \forall \ForAllInRange{i}{\NVariables}, \ForAllInRange{j}{3} \text{ s.t. } \SecPricingUMap_i(\SecPricingLexUVar) > (2/3) \SecPricingG{j}
\end{array}
\right\}.
$$
For any $\SecPricingDualVarFour \in \SecPricingSetOfDeltas(\SecPricingLexUVar)$, there exists $\SecPricingDualVarOne, \SecPricingDualVarTwo, \SecPricingDualVarThree, \SecPricingDualVarFive$ such that $(\SecPricingDualVarOne, \SecPricingDualVarTwo, \SecPricingDualVarThree, \SecPricingDualVarFour, \SecPricingDualVarFive) \in \OptSolSet\eqref{EqDual}$.
\end{lemma}

\begin{proof}[Sketch]
Let $\SecPricingLexUVar \in [0, 1]$ and $\SecPricingDualVarFour \in \SecPricingSetOfDeltas(\SecPricingLexUVar)$.
Define
\begin{align*}
\SecPricingDualVarTwo_i &= \begin{cases}
0, & \text{if $\SecPricingUMap_i(\SecPricingLexUVar) \le 1/3$}, \\
\SecPricingLexWeight^{n - i} - 4 \SecPricingDualVarFive_i, & \text{if $\SecPricingUMap_i(\SecPricingLexUVar) > 1/3$}, \\
\end{cases}
\ \
\SecPricingDualVarThree_i = \begin{cases}
0, & \text{if $\SecPricingUMap_i(\SecPricingLexUVar) \le 2/3$}, \\
\SecPricingLexWeight^{n - i} + 4 \SecPricingDualVarFive_i, & \text{if $\SecPricingUMap_i(\SecPricingLexUVar) > 2/3$}, \\
\end{cases} 
&&
\forall \ForAllInRange{i}{n},
\\
\SecPricingDualVarFive_{i} &= (3/2) (\SecPricingDualVarTwo_{i + 1} + \SecPricingDualVarThree_{i + 1} + \textstyle \sum_{j = 1}^3 \SecPricingDualVarFour_{i + 1 j}) + 3 \SecPricingDualVarFive_
{i + 1} + \SecPricingLexWeight^{n - i - 1},
&&
\forall \ForAllInRange{i}{\NVariables - 1},
\\
\SecPricingDualVarOne &= (3/2) (\SecPricingDualVarTwo_1 + \SecPricingDualVarThree_1 + \textstyle \sum_{j = 1}^3 \SecPricingDualVarFour_{1 j}) + 3 \SecPricingDualVarFive_
{1} + \SecPricingLexWeight^{n - 1}.
\end{align*}
Using $\SecPricingLexWeight \ge 40$, by induction, one can show $0 \le \SecPricingDualVarTwo_i, \SecPricingDualVarThree_i$, and $\SecPricingDualVarFive_i \le (1/4) \SecPricingLexWeight^{n - i}$.
From this, one can verify that $(\SecPricingDualVarOne, \SecPricingDualVarTwo, \SecPricingDualVarThree, \SecPricingDualVarFour, \SecPricingDualVarFive) \in \FeasSolSet\eqref{EqDual}$.
The optimality can be shown by verifying the complementary slackness condition.
\MyQED
\end{proof}

Let $\SecPricingDualVarFour'$ be such that $\SecPricingDualVarFour_{i j}' = \SecPricingDualVarFour_{i j} / \SecPricingLexWeight^{n - i}$ for all $i$ and $j$, and define $\SecPricingFrac(\SecPricingLexLVarU') = \min\{\SecPricingLexLVarU', 1 - \SecPricingLexLVarU'\}$ as the fractionality of $\SecPricingLexLVarU'$.
In Figure~\ref{fig:plotTwo}, the values of $\SecPricingLexLVarU$ and $\SecPricingDualVarFour'$ optimal for~\eqref{EqDual} are visualized in case of $\NVariables = 3$.
The shaded regions correspond to the values of $\SecPricingLexUVar$ such that there exists $i'$ such that $\SecPricingFrac(\SecPricingLexLVarU_{i'}) > 1/3$, implying $\SecPricingDualVarFour_{i' 1}' - \SecPricingDualVarFour_{i' 3}' = 1$.
At the boundaries of the shaded regions, the values of $\SecPricingDualVarFour_{i 1}' - \SecPricingDualVarFour_{i 3}'$ are not unique, and otherwise $\SecPricingDualVarFour_{i 1}' - \SecPricingDualVarFour_{i 3}' = 0$ for all $i$.
We note that in the last case (i.e., $\SecPricingDualVarFour_{i 1}' - \SecPricingDualVarFour_{i 3}' = 0$), $\SecPricingLexUVar$ belongs to
$$
\{ \SecPricingLexUVar : \SecPricingLexUVar = \BinaryVector_1 / 3 + \BinaryVector_2 / 9 + \BinaryVector_3 / 27 + \BinaryVector' / 27, \text{where } \BinaryVector \in \{0, 2\}^3, \BinaryVector' \in (0, 1) \} \cup \{ 0, 1 \},
$$
and the value of $\SecPricingDualVarFour_{i 2}'$ is binary.

\begin{figure}[btp]
\centering
\begin{tikzpicture}[
xscale=4,
yscale=4,
point/.style={circle,inner sep=0,minimum size=1mm,fill=black},
nodepadding/.style={circle,inner sep=0,minimum size=6mm},
unselected/.style={draw=black, densely dashed},
selected/.style={draw=black},
first/.style={ultra thick},
second/.style={ultra thick,blue,dashed},
third/.style={ultra thick,red,dotted},
secondSolid/.style={ultra thick,blue},
thirdSolid/.style={ultra thick,red},
firstThin/.style={densely dotted},
secondThin/.style={blue,densely dotted},
thirdThin/.style={red,densely dotted},
]

\newcommand{\LevelOne}{-0.25}
\newcommand{\LevelTwo}{-0.2}
\newcommand{\LevelThree}{-0.15}

\begin{scope}[shift={(0,0)}]

\foreach \i in {0,2,6,8} { 
  \fill[gray!50!red!20] ({(\i*3+1)/27},\LevelOne) -- ({(\i*3+1)/27},1.2) -- ({(\i*3+2)/27}, 1.2) -- ({(\i*3+2)/27}, \LevelOne) -- cycle;
}

\foreach \i in {0,2} { 
  \fill[gray!50!blue!10] ({(\i*3+1)/9},\LevelTwo) -- ({(\i*3+1)/9},1.2) -- ({(\i*3+2)/9}, 1.2) -- ({(\i*3+2)/9}, \LevelTwo) -- cycle;
}

\fill[gray!30] ({1/3},\LevelThree) -- ({1/3},1.2) -- ({2/3}, 1.2) -- ({2/3}, \LevelThree) -- cycle;

\draw[firstThin] (0, \LevelThree) -- (1, \LevelThree);

\draw[first] ({1/3},\LevelThree) -- ({2/3}, \LevelThree);

\draw[secondThin] (0, \LevelTwo) -- (1, \LevelTwo);

\foreach \i in {0,...,2} { 
  \draw[secondSolid] ({(\i*3+1)/9},\LevelTwo) -- ({(\i*3+2)/9}, \LevelTwo);
}

\draw[thirdThin] (0, \LevelOne) -- (1, \LevelOne);

\foreach \i in {0,...,8} { 
  \draw[thirdSolid] ({(\i*3+1)/27},\LevelOne) -- ({(\i*3+2)/27}, \LevelOne);
}

\node [label={left:$\SecPricingDualVarFour_{1 1}' - \SecPricingDualVarFour_{1 3}'$}] at (0,{\LevelThree + 0.04}) {};
\node [label={left:$\SecPricingDualVarFour_{2 1}' - \SecPricingDualVarFour_{2 3}'$}] at (0,\LevelTwo) {};
\node [label={left:$\SecPricingDualVarFour_{3 1}' - \SecPricingDualVarFour_{3 3}'$}] at (0,{\LevelOne - 0.04}) {};

\foreach \i in {1,...,9} {
    \draw [very thin,gray] (\i/9,-0.0) -- (\i/9,1.1);
    \draw [very thin,gray] (-0.0,\i/9) -- (1.1,\i/9);
}
   
\draw[->] (\LevelThree,0) -- (1.2,0) node[above] {$\LexUVar$};
\draw[->] (0,-0.35) -- (0,1.2) node[above] {$\LexLVarU$};

\begin{scope}
\clip(0,0) rectangle (1,1);
\draw [first] (0,0) -- (1,1);
\draw [second] (0,0) -- (1/3,1) -- (2/3, 0) -- (1,1);
\draw [third] (0,0) -- (1/9,1) -- (2/9,0) -- (3/9,1) -- (4/9,0) -- (5/9,1) -- (6/9,0) -- (7/9,1) -- (8/9, 0) -- (9/9,1);
\end{scope}

\end{scope}

\begin{scope}[shift={(1.5,0)}]

\foreach \i in {0,2,6,8} { 
  \fill[gray!50!red!20] ({(\i*3+1)/27},\LevelOne) -- ({(\i*3+1)/27},1.2) -- ({(\i*3+2)/27}, 1.2) -- ({(\i*3+2)/27}, \LevelOne) -- cycle;
}

\foreach \i in {0,2} { 
  \fill[gray!50!blue!10] ({(\i*3+1)/9},\LevelTwo) -- ({(\i*3+1)/9},1.2) -- ({(\i*3+2)/9}, 1.2) -- ({(\i*3+2)/9}, \LevelTwo) -- cycle;
}

\fill[gray!30] ({1/3},\LevelThree) -- ({1/3},1.2) -- ({2/3}, 1.2) -- ({2/3}, \LevelThree) -- cycle;

\draw[firstThin] (0, \LevelThree) -- (1, \LevelThree);

\draw[first] ({1/3},\LevelThree) -- ({2/3}, \LevelThree);

\draw[secondThin] (0, \LevelTwo) -- (1, \LevelTwo);

\foreach \i in {0,...,2} { 
  \draw[secondSolid] ({(\i*3+1)/9},\LevelTwo) -- ({(\i*3+2)/9}, \LevelTwo);
}

\draw[thirdThin] (0, \LevelOne) -- (1, \LevelOne);

\foreach \i in {0,...,8} { 
  \draw[thirdSolid] ({(\i*3+1)/27},\LevelOne) -- ({(\i*3+2)/27}, \LevelOne);
}

\foreach \i in {1,...,9} {
    \draw [very thin,gray] (\i/9,-0.0) -- (\i/9,1.1);
    \draw [very thin,gray] (-0.0,\i/9) -- (1.1,\i/9);
}
   
\draw[->] (\LevelThree,0) -- (1.2,0) node[above] {$\LexUVar$};
\draw[->] (0,-0.35) -- (0,1.2) node[above] {$\SecPricingDualVarFour_{1 2}', \SecPricingDualVarFour_{2 2}', \SecPricingDualVarFour_{3 2}'$};

\begin{scope}
\draw [first] (0,0) -- (1/2,0) -- (1/2, 1) -- (1,1);
\draw [second] (0,0) -- (1/6,0) -- (1/6, 1) -- (1/2,1) -- (1/2,0) -- (5/6,0) -- (5/6, 1) -- (1,1);
\draw [third]
(0,0) -- (1/18,0) -- (1/18, 1) -- (3/18,1) -- (3/18,0) -- (5/18,0) -- (5/18, 1) -- (7/18,1) -- (7/18, 0) -- (9/18,0) -- (9/18,1) -- (11/18,1) -- (11/18, 0) -- (13/18,0) -- (13/18,1) -- (15/18,1) -- (15/18, 0) -- (17/18,0) -- (17/18, 1);
\end{scope}

\end{scope}

\end{tikzpicture}
\caption{%
Plot of $\SecPricingLexLVarU$ and $\SecPricingDualVarFour'$.
The solid black lines represent $\LexLVarU_1$ and $\SecPricingDualVarFour_{1 2}'$, the dashed blue lines represent $\LexLVarU_2$ and $\SecPricingDualVarFour_{2 2}'$, and the dotted red lines represent $\LexLVarU_3$ and $\SecPricingDualVarFour_{3 2}'$. 
The shaded regions indicate the values of $\SecPricingLexUVar$ where $\SecPricingDualVarFour_{i 1}' - \SecPricingDualVarFour_{i 3}' > 0$ for some $i$.  
}
\label{fig:plotTwo}
\end{figure}

Given a Boolean formula $F$ with $n$ Boolean variables, consider
\begin{equation}
\MyProblemTag{EqPricingBLP:HardInstance}{(\SATBooleanFormula)}
\max_{\substack{\SecPricingLexUVar \in [0, 1], \SecPricingDualVarOne, \SecPricingDualVarTwo,\\ \SecPricingDualVarThree, \SecPricingDualVarFour, \SecPricingDualVarFive, x, y, z}}
\left\{
z - 2 \sum_{i = 1}^n (\SecPricingDualVarFour_{i 1} - \SecPricingDualVarFour_{i 3}) / \SecPricingLexWeight^{n - i} 
:
(\SecPricingDualVarOne, \SecPricingDualVarTwo, \SecPricingDualVarThree, \SecPricingDualVarFour, \SecPricingDualVarFive, x, y, z) \in \mathcal{S}\eqref{EqPricingBLPHardInstanceLowerLevel}
\right\}
\end{equation}
where
\begin{align}
\MyProblemTag{EqPricingBLPHardInstanceLowerLevel}{(\SecPricingLexUVar, \SATBooleanFormula)}
\max_{\substack{\SecPricingDualVarOne, \SecPricingDualVarTwo, \SecPricingDualVarThree, \\\SecPricingDualVarFour, \SecPricingDualVarFive, x, y, z}}
\left\{
\SecPricingLexUVar \SecPricingDualVarOne
- \frac12 \OneVector^{\top} \SecPricingDualVarTwo
- \OneVector^{\top} \SecPricingDualVarThree
- \sum_{i = 1}^n \sum_{j = 1}^{3} \SecPricingG{j} \SecPricingDualVarFour_{i j}
:
\begin{array}{ll}
A_F x + B_F y + C_F z \ge d_F, 
\\
x_i = \SecPricingDualVarFour_{i 2} / \SecPricingLexWeight^{n - i},
\qquad
\forall \ForAllInRange{i}{\NVariables},
\\
(\SecPricingDualVarOne, \SecPricingDualVarTwo, \SecPricingDualVarThree, \SecPricingDualVarFour, \SecPricingDualVarFive) \in \mathcal{F}\eqref{EqDual},
\\
0 \le \SecPricingDualVarOne, \SecPricingDualVarTwo, \SecPricingDualVarThree, \SecPricingDualVarFour, \SecPricingDualVarFive \le 10 \SecPricingLexWeight^{n - 1} \cdot \OneVector,
\\
\mathbf{0} \le x, y, z \le \mathbf{1}
\end{array}
\right\}.
\end{align}
The feasible set of~\eqref{EqPricingBLPHardInstanceLowerLevel} is independent of $\SecPricingLexUVar$.
Furthermore, the introduction of $x, y, z$ does not alter the optimal values of $(\SecPricingDualVarOne, \SecPricingDualVarTwo, \SecPricingDualVarThree, \SecPricingDualVarFour, \SecPricingDualVarFive)$, since for each $\SecPricingLexUVar \in [0, 1]$, the projection of $\FeasSolSet\eqref{EqPricingBLPHardInstanceLowerLevel}$ onto the space of $(\SecPricingDualVarOne, \SecPricingDualVarTwo, \SecPricingDualVarThree, \SecPricingDualVarFour, \SecPricingDualVarFive)$ equals $\FeasSolSet\eqref{EqDual}$.
Similarly, the variable upper bounds are loose (Lemma~\ref{lemma:SecPricing:BoundedLowerLevelFeasibleSet}).
Let
$$
\SecPricingPartialOptValFunc(\SecPricingLexUVar) :=
\max_{\substack{\SecPricingDualVarOne, \SecPricingDualVarTwo, \SecPricingDualVarThree,\\ \SecPricingDualVarFour, \SecPricingDualVarFive, x, y, z}}
\left\{
z - 2 \sum_{i = 1}^n (\SecPricingDualVarFour_{i 1} - \SecPricingDualVarFour_{i 3}) / \SecPricingLexWeight^{n - i} 
:
(\SecPricingDualVarOne, \SecPricingDualVarTwo, \SecPricingDualVarThree, \SecPricingDualVarFour, \SecPricingDualVarFive, x, y, z) \in \mathcal{S}\eqref{EqPricingBLPHardInstanceLowerLevel}
\right\}
$$
be the optimal objective value of~\eqref{EqPricingBLP:HardInstance} with $\SecPricingLexUVar$ fixed.
We have $\OptValFunc\eqref{EqPricingBLP:HardInstance} = \max_{\SecPricingLexUVar \in [0, 1]} \SecPricingPartialOptValFunc(\SecPricingLexUVar)$.
For any optimal solution for the lower-level problem, $\SecPricingDualVarFour_{i 1} - \SecPricingDualVarFour_{i 3} \ge 0$ for all $i$ (Lemma~\ref{Lemma:DualSolutionAsFlag}), and hence 
$\SecPricingPartialOptValFunc(\SecPricingLexUVar) \le 1$ for any $\SecPricingLexUVar \in [0, 1]$.
Let
\begin{equation}
\label{EqThetaOneDefinition}
\Theta_1 :=
\textstyle
\{ \SecPricingLexUVar : \SecPricingLexUVar = \sum_{i = 1}^{\NVariables} \BinaryVector_i / 3^{i} + \BinaryVector' / 3^{\NVariables}, \text{where } \BinaryVector \in \{0, 2\}^{\NVariables}, \BinaryVector' \in (0, 1) \} \cup \{ 0, 1 \},
\end{equation}
and
$$
\Theta_2 :=
\textstyle
\{ \SecPricingLexUVar : \SecPricingLexUVar = \sum_{i = 1}^{\NVariables} \BinaryVector_i / 3^{i} + \BinaryVector' / 3^{\NVariables}, \text{where } \BinaryVector \in \{0, 2\}^{\NVariables}, \BinaryVector' \in \{0, 1\} \} \setminus \{ 0, 1 \}.
$$
That is, $\Theta_2 = \mathrm{cl}(\Theta_1) \setminus \Theta_1$ be the boundary points of $\Theta_1$ that are not $0$ or $1$.

\begin{enumerate}[label=(case \arabic*)]
\item
If $\SecPricingLexUVar \in \Theta_1$, then 
$\SecPricingFrac(\SecPricingUMap_i(\SecPricingLexUVar)) < 1/3$ for all $i$ and 
the optimality of the lower-level problem implies $x = (1/2) \BinaryVector$ and $\SecPricingDualVarFour_{i 1} - \SecPricingDualVarFour_{i 3} = 0$ for all $i$, where $\BinaryVector$ is defined as in~\eqref{EqThetaOneDefinition} (or $\BinaryVector = \ZeroVector$ for $\SecPricingLexUVar = 0$ and $\BinaryVector = 2 \cdot \OneVector$ for $\SecPricingLexUVar = 1$). 
Thus, $\SecPricingPartialOptValFunc(\SecPricingLexUVar) = 1$ if $x$ satisfies $F$ and $0$ otherwise.
\item 
\label{Enum:ThetaToXRule:Two}
If $\SecPricingLexUVar \in \Theta_2$, then 
$x = (1/2) \BinaryVector$, and there exists $i'$ such that any values of $\SecPricingDualVarFour_{i' 1} - \SecPricingDualVarFour_{i' 3} \in [0, \SecPricingLexWeight^{n - i'}]$ is optimal for the lower-level problem.
Under the optimistic assumption, $\SecPricingDualVarFour_{i' 1} - \SecPricingDualVarFour_{i' 3} = 0$ can be chosen, and thus, $\SecPricingPartialOptValFunc(\SecPricingLexUVar) = 1$ if $x$ satisfies $F$ and $0$ otherwise.
\item 
Otherwise, there exists $i'$ such that optimality of the lower-level problem implies  $\SecPricingDualVarFour_{i' 1} - \SecPricingDualVarFour_{i' 3} = \SecPricingLexWeight^{n - i'}$.
Thus, $\SecPricingPartialOptValFunc(\SecPricingLexUVar) \le -1$.
\end{enumerate}
Therefore, we have $\OptValFunc\eqref{EqPricingBLP:HardInstance} = 1$ if and only if $F$ is satisfiable and $0$ otherwise.

\begin{theorem}
The decision version of~\eqref{EqPricingBLP} is NP-complete.
Moreover, this remains true even when $D_{2} = e_1^{\top}$, that is, even when the upper-level variable appears in only one lower-level cost coefficient.
\end{theorem}

\begin{proof}
Hardness follows from~\eqref{EqPricingBLP:HardInstance} and membership follows from Theorem~\ref{theorem:ComplexityOfGeneralPricingBLP}.
\MyQED
\end{proof}

It is straightforward to modify the above hardness argument, i.e. $\eqref{EqPricingBLP:HardInstance}$, to handle the pessimistic formulation.
In~\ref{Enum:ThetaToXRule:Two}, bilevel feasibility implies $\SecPricingDualVarFour_{i' 1} - \SecPricingDualVarFour_{i' 3} = \SecPricingLexWeight^{n - i'}$, and $\SecPricingPartialOptValFunc(\SecPricingLexUVar) \le -1$, but the other two cases remain the same.
In particular, the optimal objective value of the pessimistic formulation of~\eqref{EqPricingBLP:HardInstance} is $1$ if and only if $F$ is satisfiable, and $0$ otherwise.
Therefore, the decision version of the pessimistic formulation of the pricing bilevel LP is NP-hard.

\section{Conclusions}\label{SecConclusions}

In this paper, we study the computational complexity of bilevel LP with a single upper-level variable taking values in $[0,1]$.
We prove that the decision version of the problem is NP-hard, and that even determining whether a given candidate solution is globally optimal is coNP-hard.
In contrast, we show that a locally optimal solution can be found in polynomial time.
We extend the analysis to show that the decision version of the bilevel pricing problem is NP-hard, even if there is only one upper-level variable.

Beyond the complexity results, we establish three auxiliary results of independent interest. First, we bound the encoding size of an optimal solution of general bilevel LPs. Second, we provide a corrected proof on the validity of the penalty method for removing upper-level constraints, identifying the necessary assumptions under which it holds. Third, we show that, given a lexicographic LP, one can construct an equivalent LP by aggregating its objectives using weights whose encoding size is polynomial in the sizes of $c$ and $A$.
In particular, we prove that these weights are independent of the right-hand side, and are valid uniformly for all right-hand sides, a property that is crucial in the context of bilevel LPs.

\begin{appendix}

\section{Proof of Lemma~\ref{Lemma:IntegerLexicographicToWeighted}}
\label{Sec:Proof:Lemma:IntegerLexicographicToWeighted}

In this section, we prove Lemma~\ref{Lemma:IntegerLexicographicToWeighted}.
Since $\SecLexFeasibleSet'(b) = \SecLexFeasibleSet'(\lceil b \rceil)$, we may assume $b \in \Integers^{\NConstraints}$.
We begin with an integer version of the feasible direction.
\begin{definition}
\label{Def:IntegerLexicographicToWeighted:FeasibleDirection}
Let $b \in \mathbb{R}^{\NConstraints}$ and $x \in \SecLexFeasibleSet'(b)$.
We say $t$ is a feasible direction at $x$ if and only if $x + t \in \SecLexFeasibleSet'(b)$.
We denote the set of feasible directions by $\SecLexFeasibleDirections'(b, x)$.
\end{definition}
The variant of Lemma~\ref{Lemma:DecompositionByFeasibleDirection} holds for the integer case.
\begin{lemma}
\label{Lemma:IntegerLexicoraphic:DecompositionByFeasibleDirection}
For any $b \in \mathbb{R}^{\NConstraints}$ and $x, y \in \SecLexFeasibleSet'(b)$, there exists $r \in \NonNegativeIntegers$, $t_1, \ldots, t_r \in \SecLexFeasibleDirections'(b, x) $ and $\alpha_1, \ldots, \alpha_r \in \PositiveIntegers$ such that
\begin{equation}
\label{Eq:ConformalDecomposition}
y - x = \sum_{i = 1}^r \alpha_i t_i,
\quad
\| t_i \|_{\infty} \le \NVariables\SecLexMaxSubdeterminant(A), \ \forall i.
\end{equation}
\end{lemma}

We are not aware of earlier work that provides this bound, although the result may not be novel.
This type of decomposition goes back to the seminal work of Graver~\cite{graver1975foundations} (Theorem 4.5), where $A$ is assumed to have full column rank and no bound as in~\eqref{Eq:ConformalDecomposition} is given.
Since then, the literature seems to focus on the ``standard'' form, that is, equality constraints over nonnegative integers (see~\cite{weismantel1998test}, for example).
Transforming our formulation into the standard form requires slack variables, and may yield a weaker bound in~\eqref{Eq:ConformalDecomposition}.
Below, we instead provide a short proof based on the Hilbert basis.

\begin{lemma}[Hilbert basis]
\label{Lem:gordan}
Let $C' \in \Integers^{q \times n}$, $P := \{ d \in \Integers^{n} : C' d \ge 0 \}$ and $P' := \{ d \in P : \| d \|_{\infty} \le n \SecLexMaxSubdeterminant(C') \}$.
Then every $h \in P$ is a nonnegative integer combination of elements of $P'$.
\end{lemma}

For the proof of Lemma~\ref{Lem:gordan}, see \cite{Schrijver1998} (theorem~16.4).

\begin{proof}[Lemma~\ref{Lemma:IntegerLexicoraphic:DecompositionByFeasibleDirection}]
When $x = y$, the claim holds trivially.
Thus, we assume $x \ne y$.
Let $d := y - x$ and $e \in \{ -1, 1 \}^{\NConstraints}$ be such that $e_i (A d)_i \ge 0$ for every $i$. 
Let
$$
D 
:= \{ d \in \mathbb{Z}^{\NVariables} : e_i (A d)_i \ge 0 \text{ for all $i$} \}
= \{ d \in \mathbb{Z}^{\NVariables} : (E A) d \ge 0 \},
$$
where $E := \mathrm{diag}(e)$.
Since $\SecLexMaxSubdeterminant(E A) = \SecLexMaxSubdeterminant(A)$, by Lemma~\ref{Lem:gordan}, $D$ is generated by integral vectors of $l_\infty$-norm at most $\NVariables \SecLexMaxSubdeterminant(A)$.
Since $d \in D$, there exists $r$ such that
$$
y - x = \sum_{i = 1}^r \alpha_i t_i,
\qquad
\alpha_i \ge 1, \ \| t_i \|_{\infty} \le \NVariables \SecLexMaxSubdeterminant(A), \quad \forall i.
$$
We show that for any $i$, $t_i \in \SecLexFeasibleDirections'(b, x)$.
Fix $l \in \{1, \ldots, m\}$.
If $(A t_i)_l \ge 0$, then $(A(x + t_i))_l \ge (A x)_l \ge b_l$.
Otherwise, $e_l = -1$ and
$$
(A d)_l = \sum_{i' = 1}^{r} \alpha_{i'} \underbrace{(A t_{i'})_l}_{\le 0} \le (A t_{i})_l < 0,
$$
implying
$$
(A(x + t_i))_l = (A x)_l + (A t_i)_l \ge (A x)_l + (A d)_l = (A y)_l \ge b_l.
$$
Therefore, $A (x + t_i) \ge b$ and $t_i \in \SecLexFeasibleDirections'(b, x)$.
\MyQED
\end{proof}

Now, we can establish the integer version of the universal test set.
\begin{lemma}[Universal test set]
\label{lemma:IntegerLexicographic:TestSet}
Let $r \in \PositiveIntegers$ and $C' \in \Reals^{r \times \NVariables}$.
If $x' \in \SecLexFeasibleSet'(b)$ is suboptimal for $\lexmin_{x \in \SecLexFeasibleSet'(b)} \{ C' x \}$, there exists $t \in \SecLexFeasibleDirections'(b, x')$ such that $\| t \|_{\infty} \le \NVariables \SecLexMaxSubdeterminant(A)$ and $C' t \prec \ZeroVector$.
\end{lemma}

\begin{proof}
The proof is analogous to that of Lemma~\ref{lemma:TestSet}, but with Lemma~\ref{Lemma:IntegerLexicoraphic:DecompositionByFeasibleDirection} in place of Lemma~\ref{Lemma:DecompositionByFeasibleDirection}.
\end{proof}

\begin{proof}[Lemma~\ref{Lemma:IntegerLexicographicToWeighted}]
Follow the proof of Lemma~\ref{Lemma:LexicographicToWeighted}, but use Lemma~\ref{lemma:IntegerLexicographic:TestSet} in place of Lemma~\ref{lemma:TestSet}.
\end{proof}

\begin{credits}
\subsubsection{\ackname} 
This study was funded by the FRQ-IVADO Research Chair in Data Science for Combinatorial Game Theory.

\subsubsection{\discintname}
The authors have no competing interests to declare that are
relevant to the content of this article.
\end{credits}
%
%
%
\bibliographystyle{splncs04}
\bibliography{references}

@article{BasuEtAl2021,
author = {Basu, Amitabh and Ryan, Christopher Thomas and Sankaranarayanan, Sriram},
address = {Berlin/Heidelberg},
copyright = {Springer-Verlag GmbH Germany, part of Springer Nature and Mathematical Optimization Society 2019},
issn = {0025-5610},
journal = {Mathematical programming},
number = {1-2},
pages = {163-197},
publisher = {Springer Berlin Heidelberg},
title = {Mixed-integer bilevel representability},
volume = {185},
year = {2021},
doi={10.1007/s10107-019-01424-w},
}

@article{ben1990computational,
  title={Computational difficulties of bilevel linear programming},
  author={Ben-Ayed, Omar and Blair, Charles E},
  journal={Operations Research},
  volume={38},
  number={3},
  pages={556--560},
  year={1990},
  publisher={INFORMS},
  doi={10.1287/opre.38.3.556},
}

@article{Buchheim2023,
title = {Bilevel linear optimization belongs to {NP} and admits polynomial-size {KKT}-based reformulations},
journal = {Operations Research Letters},
volume = {51},
number = {6},
pages = {618-622},
year = {2023},
issn = {0167-6377},
doi = {10.1016/j.orl.2023.10.006},
author = {Christoph Buchheim},
}

@article{bui2024asymmetry,
  title={Asymmetry in the complexity of the multi-commodity network pricing problem},
  author={Bui, Quang Minh and Carvalho, Margarida and Neto, Jos{\'e}},
  journal={Mathematical Programming},
  volume={208},
  number={1},
  pages={425--461},
  year={2024},
  publisher={Springer},
  doi={10.1007/s10107-023-02043-2},
}

@article{deng1998complexity,
  title={Complexity issues in bilevel linear programming},
  author={Deng, Xiaotie},
  journal={Multilevel optimization: Algorithms and applications},
  pages={149--164},
  year={1998},
  publisher={Springer},
  doi={10.1007/978-1-4613-0307-7_6},
}

@article{ferro2026,
author = {Ferro, Leonardo and Filaci, Gianluca and Manzini, Andrea and Ornatelli, Antonio},
title = {Weighting method for non-preemptive multi-objective mathematical programming},
year = {2026},
issue_date = {May 2026},
publisher = {Kluwer Academic Publishers},
address = {USA},
volume = {94},
number = {1},
issn = {0926-6003},
doi = {10.1007/s10589-025-00762-0},
journal = {Comput. Optim. Appl.},
month = feb,
pages = {353–371},
numpages = {19},
}

@article{graver1975foundations,
  title={On the foundations of linear and integer linear programming {I}},
  author={Graver, Jack E},
  journal={Mathematical Programming},
  volume={9},
  number={1},
  pages={207--226},
  year={1975},
  publisher={Springer}
}

@article{hansen1992new,
  title={New branch-and-bound rules for linear bilevel programming},
  author={Hansen, Pierre and Jaumard, Brigitte and Savard, Gilles},
  journal={SIAM Journal on scientific and Statistical Computing},
  volume={13},
  number={5},
  pages={1194--1217},
  year={1992},
  publisher={SIAM},
  doi={10.1137/0913069},
}

@article{Jeroslow1985,
author = {Jeroslow, R. G},
address = {Heidelberg},
copyright = {1985 INIST-CNRS},
issn = {0025-5610},
journal = {Mathematical Programming},
number = {2},
pages = {146-164},
publisher = {Springer},
title = {The Polynomial Hierarchy and a Simple Model for Competitive Analysis},
volume = {32},
year = {1985},
doi = {10.1007/BF01586088},
}

@article{ketkov2025note,
  title={On the Complexity of Bilevel Linear and Quadratic Programs in Fixed Dimensions},
  author={Ketkov, Sergey S and Prokopyev, Oleg A},
  journal={arXiv preprint arXiv:2511.15592},
  year={2025}
}

@article{henke2025coupling,
  title={On coupling constraints in linear bilevel optimization},
  author={Henke, Dorothee and Lefebvre, Henri and Schmidt, Martin and Th{\"u}rauf, Johannes},
  journal={Optimization Letters},
  volume={19},
  number={3},
  pages={689--697},
  year={2025},
  publisher={Springer},
  doi={10.1007/s11590-024-02156-3},
}

@article{ignizio1984,
title = {An enhanced conversion scheme for lexicographic, multiobjective integer programs},
journal = {European Journal of Operational Research},
volume = {18},
number = {1},
pages = {57-61},
year = {1984},
issn = {0377-2217},
doi = {10.1016/0377-2217(84)90261-3},
author = {James P. Ignizio and Lyn C. Thomas},
}

@article{lenstra1983integer,
  title={Integer programming with a fixed number of variables},
  author={Lenstra Jr, Hendrik W},
  journal={Mathematics of operations research},
  volume={8},
  number={4},
  pages={538--548},
  year={1983},
  publisher={Informs}
}

@article{RodriguesEtAl2024,
  title={Unboundedness in bilevel optimization},
  author={Rodrigues, B{\'a}rbara and Carvalho, Margarida and Anjos, Miguel F and Sugishita, Nagisa},
  journal={Mathematical Programming},
  year={2026},
  publisher={Springer},
  doi={10.1007/s10107-026-02344-2},
}

@book{Schrijver1998,
author = {Schrijver, A.},
address = {Chichester},
booktitle = {Theory of Linear and Integer Programming},
isbn = {0471982326},
publisher = {Wiley},
series = {Wiley-Interscience series in discrete mathematics and optimization},
title = {Theory of Linear and Integer Programming},
year = {1998},
}

@article{sherali1983preemptive,
  title={Preemptive and nonpreemptive multi-objective programming: Relationship and counterexamples},
  author={Sherali, Hanif D and Soyster, Allen L},
  journal={Journal of Optimization Theory and Applications},
  volume={39},
  number={2},
  pages={173--186},
  year={1983},
  publisher={Springer},
  doi={10.1007/BF00934527},
}

@article{sherali1982equivalent,
  title={Equivalent weights for lexicographic multi-objective programs: Characterizations and computations},
  author={Sherali, Hanif D},
  journal={European Journal of Operational Research},
  volume={11},
  number={4},
  pages={367--379},
  year={1982},
  publisher={Elsevier},
  doi={10.1016/0377-2217(82)90202-8},
}

@InProceedings{sugishita2025complexitybilevellinearprogramming,
author="Sugishita, Nagisa and Carvalho, Margarida",
editor2="Dey, Santanu S. and Di Summa, Marco and Salvagnin, Domenico",
title="Complexity of Bilevel Linear Programming with a Single Upper-Level Variable",
booktitle="Integer Programming and Combinatorial Optimization",
year="2026",
publisher="Springer Nature Switzerlan",
pages="426--440",
}

@misc{Sugishita2026,
      title={Decision Problems in Multilevel Linear Programming}, 
      author={Nagisa Sugishita and Margarida Carvalho},
      year={2026},
      eprint={2605.04929},
      archivePrefix={arXiv},
      primaryClass={math.OC},
}

@article{vicente1994descent,
author = {Vicente, L N and Savard, G and J{\'u}dice, J},
address = {New York, NY},
copyright = {1994 INIST-CNRS},
issn = {0022-3239},
journal = {Journal of Optimization Theory and Applications},
language = {eng},
number = {2},
pages = {379-399},
publisher = {Springer},
title = {Descent Approaches for Quadratic Bilevel Programming},
volume = {81},
year = {1994},
doi = {10.1007/BF02191670},
}

@article{ProkopyevRalphs2024,
  author       = {Prokopyev, Oleg A. and Ralphs, Ted K.},
  title        = {On the Complexity of Finding Locally Optimal Solutions in Bilevel Linear Optimization},
  journal      = {Operations Research},
  year         = {2026},
  volume       = {0},
  number       = {0},
  pages        = {},
  doi          = {10.1287/opre.2024.1411},
  publisher    = {INFORMS}
}

@article{isermann1982linear,
  title={Linear lexicographic optimization},
  author={Isermann, H},
  journal={Operations-Research-Spektrum},
  volume={4},
  number={4},
  pages={223--228},
  year={1982},
  publisher={Springer}
}

@article{vavasis1990quadratic,
  title={Quadratic programming is in {NP}},
  author={Vavasis, Stephen A},
  journal={Information Processing Letters},
  volume={36},
  number={2},
  pages={73--77},
  year={1990},
  publisher={Elsevier}
}

@article{weismantel1998test,
  title={Test sets of integer programs},
  author={Weismantel, Robert},
  journal={Mathematical Methods of Operations Research},
  volume={47},
  number={1},
  pages={1--37},
  year={1998},
  publisher={Springer}
}

@article{kwek2003optimal,
  title={Optimal search for rationals},
  author={Kwek, Stephen and Mehlhorn, Kurt},
  journal={Information Processing Letters},
  volume={86},
  number={1},
  pages={23--26},
  year={2003},
  publisher={Elsevier}
}

@article{knop2020tight,
  title={Tight complexity lower bounds for integer linear programming with few constraints},
  author={Knop, Du{\v{s}}an and Pilipczuk, Micha{\l} and Wrochna, Marcin},
  journal={ACM Transactions on Computation Theory (TOCT)},
  volume={12},
  number={3},
  pages={1--19},
  year={2020},
  publisher={ACM New York, NY, USA}
}

@book{sierksma2015linear,
  title={Linear and integer optimization: theory and practice},
  author={Sierksma, Gerard and Zwols, Yori},
  year={2015},
  publisher={CRC press}
}

\end{appendix}

\end{document}